\documentclass[graybox]{svmult}

\usepackage{amsfonts,amstext,amssymb,physics} 
\usepackage{amsmath}

\usepackage{amssymb}
\usepackage{mathptmx}       
\usepackage{helvet}         
\usepackage{courier}        
\usepackage{type1cm}        
\usepackage{graphicx}        
\usepackage{multicol}        
\usepackage[bottom]{footmisc}

\usepackage{enumerate}

\usepackage{enumitem}  
\usepackage{calc}
\setlist{labelindent=1pt,itemsep=0.1cm}
\setlist[itemize]{leftmargin=0.7cm}
\setlist[enumerate]{itemindent=0em,leftmargin=0.7cm}

\usepackage{float} 

\usepackage[draft=false,hidelinks,colorlinks=true,linkcolor=black,citecolor=black,urlcolor=black,anchorcolor=blue,bookmarks=false,hypertexnames=true,backref=page]{hyperref}

\smartqed

\allowdisplaybreaks

\begin{document}
\title*{Unique common fixed points of four generalized contractive mappings in ordered partial metric spaces}
\titlerunning{Unique common fixed points of four generalized contractive mappings} 
\author{Talat Nazir \and Sergei Silvestrov}
\authorrunning{T. Nazir, S. Silvestrov} 

\institute{Talat Nazir,
\at
Department of Mathematical Sciences, University of South Africa, Florida 0003, South Africa. \\ \email{talatn@unisa.ac.za}
\and
Sergei Silvestrov
\at Division of Mathematics and Physics, School of Education, Culture and Communication, M{\"a}lardalen University, Box 883, 72123 V{\"a}ster{\aa}s, Sweden. \\  \email{sergei.silvestrov@mdu.se}
\and }
%
%


\maketitle
\label{chap:NazirSilvestrov:UCFPOPMS}

\abstract*{The existence and uniqueness of the common fixed point for generalized contractive mappings in order partial metric spaces is investigated. The existence of nonnegative solution of implicit
	nonlinear integral equations is also studied. Some examples demonstrating the validity of our
	main results are constructed. The presented results extend and unify various comparable
	results in the existing literature.
}

\abstract{The existence and uniqueness of the common fixed point for generalized contractive mappings in order partial metric spaces is investigated. The existence of nonnegative solution of implicit
	nonlinear integral equations is also studied. Some examples demonstrating the validity of our
	main results are constructed. The presented results extend and unify various comparable
	results in the existing literature.
}

\section{Introduction}
\label{secNaSe6:Introduction}
Fixed point theory is a one of the powerful tool, noteworthy and stimulating
themes of nonlinear functional analysis that blends topology, analysis and
applied mathematics. In the fixed-point technique, the controllability
problem is converted to a fixed-point problem for an applicable nonlinear
operator in a function space. An essential part of this approach is to
guarantee the solvability of the equations for an invariant subset for this operator.

Alber and Guerre-Delabrere \cite{Alber} introduced the concept of weakly
contractive mappings and proved that weakly contractive mapping defined on a
Hilbert space is a Picard operator. Later, Rhoades \cite{Rhoades} proved
that the corresponding result is also valid when Hilbert space is replaced
by a complete metric space. Dutta \textit{et al.} \cite{Dutta} generalized
the weak contractive condition and proved a fixed point theorem for a self
map, which in turn generalizes \cite[Theorem 1]{Rhoades} and the
corresponding result in \cite{Alber}. The study of common fixed points of
mappings satisfying certain contractive conditions has been at the center of
rigorous research activity. The study of common fixed point theory for
sets of several single valued maps, started with the assumption that all of
the maps commuted. Sessa \cite{Sessa} generalized the concept of commuting
maps and introducing the weakly commuting maps. Then, Jungck generalized
this idea, first to compatible mappings \cite{Jungck1} and then to weakly
compatible mappings \cite{Jungck2}. There are examples that show that each
of these generalizations of commutativity is a proper extension of the
previous definition. On the other hand, Beg and Abbas \cite{Beg} obtained a
common fixed point theorem extending weak contractive condition for two
maps. In this direction, Zhang and Song \cite{Zhang} introduced the concept
of a generalized $\varphi$-weak contraction condition and obtained a
common fixed point for two maps. Doric \cite{Doric} proved a common fixed
point theorem for generalized $(\psi ,\varphi )$-weak contractions. Abbas
and Doric \cite{Abbas} obtained a common fixed point theorem for four maps
that satisfy contractive condition which is more general than that given in
\cite{Zhang}.

In 2004, Ran and Reurings \cite{Ran} investigated the existence of fixed
points in partially ordered metric spaces, and then by Nieto and Lopez \cite{Neito}.
Further results in this direction under weak contractive condition
were proved (see for example \cite{ATS, Amini, Ciric, Harjani, Hussain, Nashine, Petrusel, Radenovic, Regan}).

In 2011, Abbas et al.  \cite{ATS} presented some common fixed point
theorems for generalized $\left(\psi ,\varphi \right)$-weakly contractive
mappings in partially ordered metric spaces. Further, Radenovi\'{c} and
Kadelburg \cite{Radenovic} proved a result for generalized weak contractive
mappings in partially ordered metric spaces.

Partial metric space is a generalized metric space in which each object does
not necessarily have to have a zero distance from itself \cite{Matthews}. A
motivation behind introducing the concept of a partial metric was to obtain
appropriate mathematical models in the theory of computation \cite{Aydi14, Heckmann, Matthews2, Schellekens}. Altun and Erduran \cite{Altun}, Oltra
and Valero \cite{Oltra} and Valero \cite{Valero} established some further
generalizations of the results in \cite{Matthews}, and Romaguera \cite{Romaguera} proved
a Caristi type fixed point theorem on partial metric spaces.
Karapinar \cite{Kara} proved some fixed point theorems for weak $\varphi$-contraction on partial metric spaces in partially ordered sets. Further results in the direction of partial metric space were proved in
\cite{AT, ATR, Aydi, Bukatin, Bukatin2, Romaguera1, WasfiSametAbbas}.

It was shown that, in some cases, the results of fixed point in partial
metric spaces can be obtained directly from their induced metric
counterparts \cite{HRS13, JKS13, SVV13}. However, some conclusions important
for the application of partial metrics in information sciences cannot be
obtained in this way. For example, if $u$ is a fixed point of map $f$, then,
by using the method from \cite{HRS13}, we cannot conclude that $p(fu,fu)=0=p(u,u)$. For further details, we refer the reader to \cite{NK13,NKR13}.

Our aim is to study the unique common fixed point results for four mappings
satisfying generalized contractive conditions in the setup of ordered
partial metric spaces.

In the sequel, $\mathbb{R},$ $\mathbb{R}_{\geq 0}$ and $\mathbb{Z}_{\geq 0}$ will denote the set of all real numbers, the set of all nonnegative real
numbers and the set of all non-negative integers, respectively. The usual order
on $\mathbb{R}$ (respectively, on $\mathbb{R}_{\geq 0})$ will be indistinctly
denoted by $\leq $ or $\geq$.

Consistent with \cite{Altun} and \cite{Matthews}, the following definitions
and results will be needed in the sequel.\newline

\begin{definition}
	\label{NaSeDefinition1.1.}
	Let $X$ be a nonempty set. A function
	$p:X\times X\rightarrow \mathbb{R}_{\geq 0}$ is said to be a partial metric on $X$
	if for any $x,y,z\in X,$ the following conditions hold true:
\begin{enumerate}[label=\textup{\arabic*)}, ref=\arabic*]
		\item   $p(x,x)=p(y,y)=p(x,y)$ if and only if $x=y;$
		
		\item  $p(x,x)\leq p(x,y);$
		
		\item  $p(x,y)=p(y,x);$
		
		\item   $p(x,z)\leq p(x,y)+p(y,z)-p(y,y)$.
	\end{enumerate}	
 The pair $(X,p)$ is then called a partial metric space.
\end{definition}
 If $p(x,y)=0$, then (${1}$) and (${2}$) imply that $x=y$. But
the converse does not always hold.

 A trivial example of a partial metric space is the pair $(\mathbb{R}_{\geq 0},p)$, where the partial metric $p:\mathbb{R}_{\geq 0}\times \mathbb{R}_{\geq 0}\rightarrow \mathbb{R}_{\geq 0}$ is
defined as $p(x,y)=\max \{x,y\}$.

\begin{example} (\cite{Matthews})\
	\label{NaSeExam1.2.}
	 If $X=\{[a,b]:a,b\in \mathbb{
		R},a\leq b\},$ then $$p([a,b],[c,d])=\max \{b,d\}-\min \{a,c\}$$ defines a
	partial metric $p$ on $X$.
\end{example}
 For some more examples of partial metric spaces, we refer to \cite{AT, ATR,
	Altun, Bukatin2, Romaguera, Schellekens}.\smallskip

Each partial metric $p$ on $X$ generates a $T_{0}$ topology $\tau _{p}$ on $
X $ which has as a base the family open $p$-balls $\{B_{p}(x,\varepsilon
):x\in X,\varepsilon >0\},$ where $B_{p}(x,\varepsilon )=\{y\in
X:p(x,y)<p(x,x)+\varepsilon \},$ for all $x\in X$ and $\varepsilon >0$.\smallskip

Observe (see \cite[p. 187]{Matthews}) that a sequence $\{x_{n}\}$ in a
partial metric space $X$ converges to a point $x\in X$, with respect to $
\tau _{p},$ if and only if $p(x,x)=\lim \limits_{n\rightarrow \infty
}p(x,x_{n})$.\smallskip

If $p$ is a partial metric on $X$, then the function $p^{S}:X\times
X\rightarrow \mathbb{R}_{\geq 0}$ given by $p^{S}(x,y)=2p(x,y)-p(x,x)-p(y,y)$
defines a metric on $X$.

Furthermore, a sequence $\{x_{n}\}$ converges in $(X,p^{S})$ to a point
$x\in X$ if and only if
\begin{equation}
	\lim_{n,m\rightarrow \infty }p(x_{n},x_{m})=\lim_{n\rightarrow \infty}p(x_{n},x)=p(x,x).  \notag
\end{equation}

\begin{definition}[\cite{Matthews}]
	\label{NaSeDefinition1.3.} A sequence $
	\{x_{n}\}$ in a partial metric space $X$ is said to be a Cauchy sequence if $
	\lim \limits_{n,m\rightarrow \infty }p(x_{n},x_{m})$ exists and is finite.
\end{definition}
 A partial metric space $X$ is said to be complete if every Cauchy sequence $
\{x_{n}\}$ in $X$ converges with respect to $\tau _{p}$ to a point $x\in X$
such that $\lim \limits_{n\rightarrow \infty }p(x,x_{n})=p(x,x)$. In this
case, we say that the partial metric $p$ is complete.
\begin{lemma}[\cite{Altun,Matthews}]
	\label{NaSeLemma1.4.}
	Let $X$ be a
	partial metric space.
	
	\begin{enumerate}[label=\textup{(\roman*)}, ref=(\roman*)]
		\item A sequence $\{x_{n}\}$ in $X$ is a Cauchy sequence in $X$ if and
		only if it is a Cauchy sequence in metric space $(X,p^{S})$.
		\item A partial metric space $(X,p)$ is complete if and only if the
		metric space $(X,p^{S})$ is complete.
	\end{enumerate}
\end{lemma}

Two self maps $f$ and $g$ on $X$ are said to be compatible if, whenever $
\{x_{n}\}$ in $X$ such that $\lim \limits_{n\rightarrow \infty
}p^{S}(fx_{n},x)=0$ and$\  \lim \limits_{n\rightarrow \infty
}p^{S}(gx_{n},x)=0 $ for some $x\in X,$ then $\lim \limits_{n\rightarrow
	\infty }p^{S}(fgx_{n},gfx_{n})=0.$

If\ $fx=gx$ for some $x$ in $X$, then $x$ is called a coincidence point of $
f $ and $g$. Furthermore, if the mappings are commuting on their coincidence
point, then such mappings are called weakly compatible, \cite{Jungck2}.

\begin{definition}
	\label{NaSeDefinition1.5.}
	Let $X$ be a nonempty set. Then $
	(X,\preceq ,p)$ is called an ordered partial metric space if
\begin{enumerate}[label=\textup{(\roman*)}, ref=(\roman*)]
		\item $p$ is a partial metric on $X,$
\item $\preceq $ is a partial order on $X$.
\end{enumerate}
\end{definition}
We say that the elements $x,y\in X$ are called comparable if either $
x\preceq y$ or $y\preceq x$ holds.

\begin{definition}[\cite{ATS}]
	\label{NaSeDefinition1.6.} Let $(X,\preceq )$ be a
	partially ordered set and $f\ $and $g$ be two self-maps of $X.$ Mapping $f$
	is said to be dominated if $fx\preceq x$ for each $x$ in $X$. A mapping $g$
	is said to be dominating if $x\preceq gx$ for each $x$ in $X.$
\end{definition}

\begin{example}
	\label{NaSeExam1.7.}
	Let $X=[0,1]$ be endowed with usual
	ordering. Let $f,g:X\rightarrow X$ defined by $fx=\dfrac{x}{k}$ and $gx=kx$
	for any positive real number $k\geq 1.$ It is easy to see that $f$ is
	dominated and $g$ is a dominating map.
\end{example}

Zhang and Song \cite{Zhang} obtained the following common fixed point result
in metric spaces for a generalized $\varphi$-weak contraction.

\begin{theorem}[\cite{Zhang}]
	\label{NaSeTheorem1.8.} Let $(X,d)$ be a complete
	metric space, and let $f,g:X\rightarrow X$ be two self-mappings such that
	for all $x,y\in X,$ $d(fx,gy)\leq M(x,y)-\varphi (M(x,y))$ holds, where $
	\varphi :[0,\infty )\rightarrow \lbrack 0,\infty )$ is a lower
	semi-continuous function with $\varphi \left( t\right) >0$ for $t\in \left(
	0,\infty \right) $, $\varphi (0)=0,$ and
	\begin{equation*}
		M(x,y)=\max \{d(x,y),d(fx,x),d(gy,y),\dfrac{d(x,gy)+d(fx,y)}{2}\}.
	\end{equation*}
	Then there exists a unique point $u\in X$ such that $u=fu=gu$.
\end{theorem}
Aydi in \cite{Aydi} obtained the following result in partial metric spaces
endowed with a partial order.

\begin{theorem}
	\label{NaSeTheorem1.9.}  Let $(X,\leq_{X})$ be a partially
	ordered set and let $p$ be a partial metric on $X$ such that $(X,p)$ is
	complete. Let $f:X\rightarrow X$ be a nondecreasing map with respect to $
	\leq _{X}$. Suppose that the following conditions hold for $y\leq_{X} x$: 	
\begin{enumerate}[label=\textup{(\roman*)}, ref=(\roman*)]
		\item the inequality holds
$$
			p(fx,fy)\leq p(x,y)-\varphi (p(x,y)),  		
$$
where $\varphi :[0,\infty ) \rightarrow \lbrack 0,\infty ) $ is
		a continuous and non-decreasing function such that it is positive in $
		(0,\infty ) $, $\varphi (0)=0$ and $\underset{t\rightarrow \infty }{
			\lim }\varphi (t)=\infty ;$
		
		\item there exist $x_{0}\in X$ such that $x_{0}\leq _{X}fx_{0}$;
		
		\item $f$ is continuous in $(X,p)$, or; \\ if a non-decreasing sequence $\{x_{n}\}$ converges to $x\in X$
		, then $x_{n}\leq _{X}x$ for all $n$.
	\end{enumerate}
	
	Then $f$ has a fixed point $u\in X$. Moreover, $p(u,u)=0$.
\end{theorem}

\begin{definition}[\cite{Doric}]
	\label{NaSedefinition1.10.} The control functions\ $
	\psi $\ and $\varphi $\ are defined as
	
\begin{enumerate}[label=\textup{\arabic*)}, ref=\arabic*]
		\item $\psi :[0,\infty )\rightarrow \lbrack 0,\infty )$ is a continuous
		nondecreasing function with $\psi (t)=0$ if and only if $t=0$,
		
		\item $\varphi :[0,\infty )\rightarrow \lbrack 0,\infty )$ is a lower
		semi-continuous function with $\varphi (t)=0$ if and only if $t=0.$
	\end{enumerate}
\end{definition}

A subset $W$ of a partially ordered set $X$ is said to be well ordered if
every two elements of $W$ are comparable.

Recently, Abbas et al. \cite{AT} obtained the following result in partial
metric spaces.

\begin{theorem}
	\label{NaSeTheorem1.11} Let $(X,\preceq )$ be a partially
	ordered set such that there exist a complete partial metric $p$ on $X$ and $
	f $ a nondecreasing self map on $X$. Suppose that for every two elements $
	x,y\in X$ with $y\preceq x,$ we have
	\begin{align*}
		& \psi (p(fx,fy))\leq \psi (M(x,y))-\phi (M(x,y)),  \\ 	
        & M(x,y)=\max \{p(x,y),p(fx,x),p(fy,y),\dfrac{p(x,fy)+p(y,fx)}{2}\},
	\end{align*}
	where $\psi $ and $\phi $ are control functions. If there exists $x_{0}\in X$
	with $x_{0}\preceq fx_{0}$ and one of the following two conditions is
	satisfied:
	\begin{enumerate}[label=\textup{(\roman*)}, ref=(\roman*)]
		\item $f$ is continuous self map on $(X,p^{S})$,		
		\item for any nondecreasing sequence $\{x_{n}\}$ in $(X,\preceq )$ with
		$\lim \limits_{n\rightarrow \infty }p^{S}(z,x_{n})=0$ it follows $x_{n}\preceq z$ for all $n\in
		\mathbb{Z}_{\geq 0},$
	\end{enumerate}
then $f$ has a fixed point. Moreover, the set of fixed points of $f $ is well ordered if and only if $f$ has one and only one fixed point.
\end{theorem}

\section{Common Fixed Point Results}

In this section, we obtain common fixed point theorems for four mappings
defined on an ordered partial metric space. We start with the following result.

\begin{theorem}
	\label{NaSeTheorem2.1.}
	Let $(X,\preceq,p)$ be an ordered complete partial metric space. Let
$f,g,S$ and $T$ be self maps on $X$, $(f,g)$ be the pair of dominated and $(S,T)$ be the pair of
		dominating maps with $f(X) \subseteq T(X)$ and
$g(X) \subseteq S(X) $. Suppose that, there
		exists control functions $\psi $ and $\varphi $ such
		that, for every two comparable elements $x,y\in X$,
	\begin{equation}
		\psi (p(fx,gy))\leq \psi (M_{p}(x,y))-\varphi (M_{p}(x,y)), \label{t.OMSeq1.1}
	\end{equation}
	is satisfied where
	\begin{equation*}
		M_{p}(x,y)=\max \{p(Sx,Ty),p(fx,Sx),p(gy,Ty),\dfrac{p(Sx,gy)+p(fx,Ty)}{2}\}.
	\end{equation*}
	If for any nonincreasing sequence $\{x_{n}\}$ in $(X,\preceq )$ with $x_{n}\preceq y_{n}$ for all $n$
	and $\lim \limits_{n\rightarrow \infty }p^{S}(x_{n},u)=0$
	it holds that $u\preceq y_{n}$ for all $n\in\mathbb{Z}_{\geq 0}$ and either of the following conditions hold:
		\begin{enumerate}[label=\textup{(\roman*)}, ref=(\roman*)]
		\item \label{condiNaSeTheorem2.1.} $\{f,S\}$ are compatible, $f$ or $S$ is continuous on $(X,p^{S})$ and $\{g,T\}$ are weakly
		compatible;
	\item \label{condiiNaSeTheorem2.1.} $\{g,T\}$ are compatible, $g$ or $T$ is continuous on $(X,p^{S})$ and $\{f,S\}$ are weakly
		compatible,
	\end{enumerate}
	 then $f$,$g$,$S$ and $T$ have a common fixed point. Moreover, the set of common fixed points of $f$,
	$g$, $S$ and $T$ is well ordered if and only if $f$, $g$, $S$ and $T$
		have one and only one\ common fixed point.
\end{theorem}

\begin{proof} \smartqed Let $x_{0}$ be an arbitrary point in $
	X.$ We construct sequences $\{x_{n}\}$ and $\{y_{n}\}$\ in $X$ such that $
	y_{2n+1}=fx_{2n}=Tx_{2n+1},$ and $y_{2n+2}=gx_{2n+1}=Sx_{2n+2}.$ By given
	assumptions, $x_{2n+2}\preceq Sx_{2n+2}=gx_{2n+1}\preceq x_{2n+1},$ and $
	x_{2n+1}\preceq Tx_{2n+1}=fx_{2n}\preceq x_{2n}.$ Thus, for all $n\in
	\mathbb{Z}_{\geq 0}$ we have $x_{n+1}\preceq x_{n}.$ We suppose that $p(y_{2n},y_{2n+1})>0,$
	for every $n.$ If not, then $y_{2n}=y_{2n+1},$ for some $n$. Further, since $
	x_{2n}$\ and $x_{2n+1}$ are comparable, so from (\ref{t.OMSeq1.1}), we have
	\begin{multline*}
		\psi (p(y_{2n+1},y_{2n+2}))=\psi (p(fx_{2n},gx_{2n+1}))\\
\leq \psi(M_{p}(x_{2n},x_{2n+1}))-\varphi (M_{p}(x_{2n},x_{2n+1})), 
	\end{multline*}
	where \
	\begin{eqnarray*}
		&&M_{p}(x_{2n},x_{2n+1})=\max
		\{p(Sx_{2n},Tx_{2n+1}),p(fx_{2n},Sx_{2n}),p(gx_{2n+1},Tx_{2n+1}), \\
		&& \hspace{4cm} \dfrac{p(Sx_{2n},gx_{2n+1})+p(fx_{2n},Tx_{2n+1})}{2}\} \\
		&=&\max \{p(y_{2n},y_{2n+1}),p(y_{2n+1},y_{2n}),p(y_{2n+2},y_{2n+1}), \\
		&& \hspace{3cm} \dfrac{p(y_{2n},y_{2n+2})+p(y_{2n+1},y_{2n+1})}{2}\} \\
		&=&p(y_{2n+1},y_{2n+2}).
	\end{eqnarray*}
	Hence, $\psi (p(y_{2n+1},y_{2n+2}))\leq \psi (p(y_{2n+1},y_{2n+2}))-\varphi
	(p(y_{2n+1},y_{2n+2}))$ implies that
	 $\varphi (p(y_{2n+1},y_{2n+2}))=0.$ As, $\varphi (t)=0$ if and
	only if $t=0$, it follows that $y_{2n+1}=y_{2n+2}.$ Following the similar
	arguments, we get $y_{2n+2}=y_{2n+3}$ and so on. Thus\ $y_{2n}$ is the
	common fixed point of $f,$ $g,$ $S$ and $T$ as $\{y_{n}\}$ became a constant
	sequence in $X$.
	
	 Taking $p(y_{2n},y_{2n+1})>0$ for each $n.$ Since $x_{2n}$ and $
	x_{2n+1}$ are comparable, from (\ref{t.OMSeq1.1}), we obtain that
	\begin{equation*}
		\psi (p(y_{2n+2},y_{2n+1}))=\psi (p(y_{2n+1},y_{2n+2}))=\psi
		(p(fx_{2n},gx_{2n+1}))
	\end{equation*}
	\begin{equation}
		\leq \psi (M_{p}(x_{2n},x_{2n+1}))-\varphi (M_{p}(x_{2n},x_{2n+1})),
		\label{t.OMSeq1.3}
	\end{equation}%
	where
	\begin{eqnarray*}
		&&M_{p}(x_{2n},x_{2n+1})=\max \{p(Sx_{2n},Tx_{2n+1}),p(fx_{2n},Sx_{2n}), \\
		&&\hspace{2cm} p(gx_{2n+1},Tx_{2n+1}),\dfrac{
			p(Sx_{2n},gx_{2n+1})+p(fx_{2n},Tx_{2n+1})}{2}\} \\
		&=&\max \{p(y_{2n},y_{2n+1}),p(y_{2n+1},y_{2n}),\\
            && \hspace{2cm} p(y_{2n+2},y_{2n+1}), \dfrac{p(y_{2n},y_{2n+2})+p(y_{2n+1},y_{2n+1})}{2}\} \\
		&\leq &\max \{p(y_{2n+1},y_{2n}),p(y_{2n+2},y_{2n+1}),\dfrac{
			p(y_{2n},y_{2n+1})+p(y_{2n+1},y_{2n+2})}{2}\} \\
		&=&\max \{p(y_{2n+1},y_{2n}),p(y_{2n+2},y_{2n+1})\}.
	\end{eqnarray*}
	If $\max \{p(y_{2n+1},y_{2n}),p(y_{2n+2},y_{2n+1})\}=p(y_{2n+2},y_{2n+1}),$
	then $$M_{p}(x_{2n},x_{2n+1})\leq p(y_{2n+2},y_{2n+1}).$$ But $
	M_{p}(x_{2n},x_{2n+1})\geq p(y_{2n+2},y_{2n+1}),$ and so
	\begin{equation*}
		M_{p}(x_{2n},x_{2n+1})=p(y_{2n+2},y_{2n+1}),
	\end{equation*}
	and (\ref{t.OMSeq1.3}) give
	\begin{eqnarray*}
		\psi (p(y_{2n+2},y_{2n+1})) &\leq &\psi (M_{p}(x_{2n},x_{2n+1}))-\varphi
		(M_{p}(x_{2n},x_{2n+1})) \\
		&=&\psi (p(y_{2n+2},y_{2n+1}))-\varphi (p(y_{2n+2},y_{2n+1})),
	\end{eqnarray*}
	a contradiction. Hence $p(y_{2n+2},y_{2n+1})\leq p(y_{2n+1},y_{2n})$.
	Moreover $M_{p}(x_{2n},x_{2n+1})\leq p(y_{2n},y_{2n+1}).$ But, since $
	M_{p}(x_{2n},x_{2n+1})\geq p(y_{2n},y_{2n+1}),$ it follows that
$$M_{p}(x_{2n},x_{2n+1})=p(y_{2n},y_{2n+1}).$$
Similarly, $p(y_{2n+3},y_{2n+2})\leq p(y_{2n+2},y_{2n+1}).$ Thus, the
	sequence $\{p(y_{2n+1},y_{2n})\}$ is nonincreasing. Hence, there exists $
	c\geq 0$ such that $\lim \limits_{n\rightarrow \infty }p(y_{2n+1},y_{2n})=c.$
	Suppose that $c>0.$ Then, $\psi (p(y_{2n+2},y_{2n+1}))\leq \psi
	(M_{p}(x_{2n+1},x_{2n}))-\varphi (M_{p}(x_{2n+1},y_{2n})),$ and by lower
	semicontinuity of $\varphi ,$ we have
	\begin{equation*}
		\underset{n\rightarrow \infty }{\lim \sup }\ \psi (p(y_{2n+2},y_{2n+1}))\leq
		\underset{n\rightarrow \infty }{\lim \sup }\ \psi (p(y_{2n},y_{2n+1}))-%
		\underset{n\rightarrow \infty }{\lim \inf }\ \varphi (p(y_{2n},y_{2n+1})),
	\end{equation*}
	which implies that $\psi (c)\leq \psi (c)-\varphi (c),$ a contradiction.
	Therefore $c=0.$ So we conclude that%
	\begin{equation}
		\lim_{n\rightarrow \infty }p(y_{2n+1},y_{2n})=0.  \label{t.OMSeq1.4}
	\end{equation}%
	Now, we show that $\lim \limits_{n,m\rightarrow \infty }p(y_{2n},y_{2m})=0$.
	If not, there is $\varepsilon >0,$ and there exist even integers $2n_{k}$
	and $2m_{k}$ with $2m_{k}>2n_{k}>k$ such that%
	\begin{equation}
		p(y_{2m_{k}},y_{2n_{k}})\geq \varepsilon ,  \label{t.OMSeq1.5}
	\end{equation}%
	and $p(y_{2m_{k}-2},y_{2n_{k}})<\varepsilon .$ Since
	\begin{eqnarray*}
		\varepsilon &\leq &p(y_{2m_{k}},y_{2n_{k}}) \\
		&\leq
		&p(y_{2n_{k}},y_{2m_{k}-2})+p(y_{2m_{k}-2},y_{2m_{k}})-p(y_{2m_{k}-2},y_{2m_{k}-2})
		\\
		&\leq
		&p(y_{2n_{k}},y_{2m_{k}-2})+p(y_{2m_{k}-2},y_{2m_{k}-1})+p(y_{2m_{k}-1},y_{2m_{k}})
		\\
		&&-p(y_{2m_{k}-1},y_{2m_{k}-1})-p(y_{2m_{k}-2},y_{2m_{k}-2}),
	\end{eqnarray*}
	from (\ref{t.OMSeq1.4}) and (\ref{t.OMSeq1.5}), we have
	\begin{equation}
		\lim_{k\rightarrow \infty }p(y_{2m_{k}},y_{2n_{k}})=\varepsilon.  \label{t.OMSeq1.6}
	\end{equation}
	Also (\ref{t.OMSeq1.6}) and inequality
$$p(y_{2m_{k}},y_{2n_{k}})\leq
	p(y_{2m_{k}},y_{2m_{k}-1})+p(y_{2m_{k}-1},y_{2n_{k}})-p(y_{2m_{k}-1},y_{2m_{k}-1})
	$$
give that $\varepsilon \leq \lim \limits_{k\rightarrow \infty
	}p(y_{2m_{k}-1},y_{2n_{k}}),$ while inequality
$$ p(y_{2m_{k}-1},y_{2n_{k}})\leq
	p(y_{2m_{k}-1},y_{2m_{k}})+p(y_{2m_{k}},y_{2n_{k}})-p(y_{2m_{k}},y_{2m_{k}})$$
	yields $\lim \limits_{k\rightarrow \infty }p(y_{2m_{k}-1},y_{2n_{k}})\leq
	\varepsilon,$ and hence
	\begin{equation}
		\lim_{k\rightarrow \infty }p(y_{2m_{k}-1},y_{2n_{k}})=\varepsilon.
		\label{t.OMSeq1.7}
	\end{equation}
	Now (\ref{t.OMSeq1.7}) and inequality $$p(y_{2m_{k}-1},y_{2n_{k}})\leq
	p(y_{2m_{k}-1},y_{2n_{k}+1})+p(y_{2n_{k}+1},y_{2n_{k}})-p(y_{2n_{k}+1},y_{2n_{k}+1})
	$$ give $\varepsilon \leq \lim \limits_{k\rightarrow \infty
	}p(y_{2m_{k}-1},y_{2n_{k}+1}),$ while inequality
$$p(y_{2m_{k}-1},y_{2n_{k}+1})\leq
	p(y_{2m_{k}-1},y_{2n_{k}})+p(y_{2n_{k}},y_{2n_{k}+1})-p(y_{2n_{k}},y_{2n_{k}})
	$$
yields $\lim \limits_{k\rightarrow \infty
	}p(y_{2m_{k}-1},y_{2n_{k}+1})\leq \varepsilon ,$ and so
	\begin{equation*}
		\lim_{k\rightarrow \infty }p(y_{2m_{k}-1},y_{2n_{k}+1})=\varepsilon .
	\end{equation*}
	As
	\begin{align*}
		& M_{p}(x_{2n_{k}},x_{2m_{k}-1})=\max
		\{p(Sx_{2n_{k}},Tx_{2m_{k}-1}),p(fx_{2n_{k}},Sx_{2n_{k}}), \\
		& \hspace{1cm} p(gx_{2m_{k}-1},Tx_{2m_{k}-1}),\frac{
			p(Sx_{2n_{k}},gx_{2m_{k}-1})+p(fx_{2n_{k}},Tx_{2m_{k}-1})}{2}\} \\
		&=\max \{p(y_{2n_{k}},y_{2m_{k}-1}),p(y_{2n_{k}+1},y_{2n_{k}}), \\
		& \hspace{1cm} p(y_{2m_{k}},y_{2m_{k}-1}),\frac{
			p(y_{2n_{k}},y_{2m_{k}})+p(y_{2n_{k}+1},y_{2m_{k}-1})}{2}\}.
	\end{align*}
	So, $\lim \limits_{k\rightarrow \infty }M_{p}(x_{2n_{k}},x_{2m_{k}-1})=\max
	\{ \varepsilon ,0,0,\varepsilon \}=\varepsilon .$ From (\ref{t.OMSeq1.1}), we obtain
	\begin{multline*}
		\psi (p(y_{2n_{k}+1},y_{2m_{k}}))=\psi (p(fx_{2n_{k}},gx_{2m_{k}-1})) \\
        \leq \psi (M_{p}(x_{2n_{k}},x_{2m_{k}-1}))-\varphi
		(M_{p}(x_{2n_{k}},x_{2m_{k}-1})).
	\end{multline*}
	Taking upper limit as $k\rightarrow \infty $ implies that $\psi (\varepsilon
	)\leq \psi (\varepsilon )-\varphi (\varepsilon ),$ a contradiction as $
	\varepsilon >0$. Thus, we obtain $\lim \limits_{n,m\rightarrow \infty
	}p(y_{2n},y_{2m})=0,$ and it follows that $\{y_{2n}\}$ is a Cauchy sequence
	in $(X,p),$ and hence Cauchy in $(X,p^{S})$ by Lemma \ref{NaSeLemma1.4.}. Since $(X,p)$ is
	complete, it follows from Lemma \ref{NaSeLemma1.4.}, $(X,p^{S})$ is also complete, so the
	sequence $\{y_{2n}\}$ is convergent in the metric space $(X,p^{S}).$
	Therefore, there exists a point $z$ in $X$ such that $\lim
	\limits_{n\rightarrow \infty }p^{S}(y_{2n},z)=0.$ Hence,
	\begin{align*}
	& \lim_{n\rightarrow \infty }y_{2n+1}=\lim_{n\rightarrow \infty}Tx_{2n+1}=\lim_{n\rightarrow \infty }fx_{2n}=z,
\\
	& \lim_{n\rightarrow \infty }y_{2n+2}=\lim_{n\rightarrow \infty}Sx_{2n+2}=\lim_{n\rightarrow \infty }gx_{2n+1}=z.
	\end{align*}
	Equivalently, we have $\lim \limits_{n,m\rightarrow \infty
	}p(y_{2n},y_{2m})=\lim \limits_{n\rightarrow \infty }p(y_{2n},z)=p(z,z).$
	
	Assume that $S$ is continuous on $(X,p^{S})$. Then $\lim
	\limits_{n\rightarrow \infty }SSx_{2n+2}=\lim \limits_{n\rightarrow \infty
	}Sfx_{2n+2}=Sz.$ Also, since $\{f,S\}$ are compatible, we have
	\begin{equation*}
		\lim_{n\rightarrow \infty }fSx_{2n+2}=\lim_{n\rightarrow \infty
		}Sfx_{2n+2}=Sz.
	\end{equation*}
	As, $Sx_{2n+2}=gx_{2n+1}\preceq x_{2n+1},$ so from (\ref{t.OMSeq1.1}), we have
	\begin{equation}
		\psi (p(fSx_{2n+2},gx_{2n+1}))\leq \psi \left(
		M_{p}(Sx_{2n+2},x_{2n+1})\right) -\varphi (M_{p}(Sx_{2n+2},x_{2n+1})),
		\label{t.OMSeq1.9}
	\end{equation}%
	where
	\begin{multline*}
		M_{p}(Sx_{2n+2},x_{2n+1}) =\max
		\{p(SSx_{2n+2},Tx_{2n+1}),p(fSx_{2n+2},SSx_{2n+2}), \\
		p(gx_{2n+1},Tx_{2n+1}),\frac{p(SSx_{2n+2},gx_{2n+1})+p(fSx_{2n+2},Tx_{2n+1})}{2}\}.
	\end{multline*}
	Now we show that $\lim \limits_{n\rightarrow \infty
	}p(fSx_{2n+2},gx_{2n+1})=p(Sz,z).$ Indeed,
	\begin{multline*}
p^{S}(fSx_{2n+2},gx_{2n+1})=\\ 2p(fSx_{2n+2},gx_{2n+1})-p(fSx_{2n+2},fSx_{2n+2})
-p(gx_{2n+1},gx_{2n+1}),
	\end{multline*}
	implies
	\begin{multline*}
		p(fSx_{2n+2},fSx_{2n+2})+p(gx_{2n+1},gx_{2n+1})+p^{S}(fSx_{2n+2},gx_{2n+1})\\
=2p(fSx_{2n+2},gx_{2n+1}),
	\end{multline*}
	which on taking limit as $n\rightarrow \infty ,$ implies that%
	\begin{equation*}
		p(Sz,Sz)+p(z,z)+p^{S}(Sz,z)=2\lim_{n\rightarrow \infty
		}p(fSx_{2n+2},gx_{2n+1}).
	\end{equation*}
	This further implies that%
	\begin{equation*}
		p(Sz,Sz)+p(z,z)+[2p(Sz,z)-p(Sz,Sz)-p(z,z)]=2\lim_{n\rightarrow \infty
		}p(fSx_{2n+2},gx_{2n+1}),
	\end{equation*}
	that is,%
	\begin{equation*}
		p(Sz,z)=\lim_{n\rightarrow \infty }p(fSx_{2n+2},gx_{2n+1}).
	\end{equation*}
	From (\ref{t.OMSeq1.9}), on taking upper limit as $n\rightarrow \infty ,$ we obtain
$$\psi(p(Sz,z))\leq \psi \left( p(Sz,z)\right) -\varphi (p(Sz,z)),$$ and $Sz=z.$
	
	Now, as $gx_{2n+1}\preceq x_{2n+1}$ and $gx_{2n+1}\rightarrow z$ as $
	n\rightarrow \infty ,$ it follows that $z\preceq x_{2n+1}$. Hence from
	(\ref{t.OMSeq1.1}), we have
	\begin{equation*}
		\psi (p(fz,gx_{2n+1}))\leq \psi \left( M_{p}(z,x_{2n+1})\right) -\varphi
		(M_{p}(z,x_{2n+1})),  
	\end{equation*}
	where
	\begin{align*}
		M_{p}(z,x_{2n+1}) &=\max \{p(Sz,Tx_{2n+1}),p(fz,Sz),p(gx_{2n+1},Tx_{2n+1}),
		\\
		&\hspace{2cm} \frac{p(Sz,gx_{2n+1})+p(fz,Tx_{2n+1})}{2}\} \\
		&=\max \{p(z,Tx_{2n+1}),p(fz,z),p(gx_{2n+1},Tx_{2n+1}), \\
		&\hspace{2cm} \frac{p(z,gx_{2n+1})+p(fz,Tx_{2n+1})}{2}\}.
	\end{align*}
	On taking upper limit as $n\rightarrow \infty ,$ we have $\psi (p(fz,z))\leq
	\psi \left( p(fz,z)\right) -\varphi (p(fz,z)),$ and $fz=z.$
	
	Since $f(X)\subseteq T(X),$ there exists a point $w\in X$ such that $fz=Tw.$
	Suppose that $gw\neq Tw.$ Since $w\preceq Tw=fz\preceq z\ $implies $w\preceq
	z.$ From (\ref{t.OMSeq1.1}), we obtain
	\begin{equation}
		\psi (p(Tw,gw))=\psi (p(fz,gw))\leq \psi \left( M_{p}(z,w)\right) -\varphi
		(M_{p}(z,w)),  \label{t.OMSeq1.11}
	\end{equation}
	\begin{eqnarray*}
		\text{where }M_{p}(z,w) &=&\max \{p(Sz,Tw),p(fz,Sz),p(gw,Tw),\frac{
			p(Sz,gw)+p(fz,Tw)}{2}\} \\
		&=&\max \{p(z,z),p(z,z),p(gw,Tw),\frac{p(Tw,gw)+p(Tw,Tw)}{2}\} \\
		&=&p(Tw,gw).
	\end{eqnarray*}
	Now (\ref{t.OMSeq1.11}) becomes $\psi (p(Tw,gw))\leq \psi \left( p(Tw,gw)\right) -\varphi
	(p(Tw,gw)),$ a contradiction. Hence, $Tw=gw.$ Since $g$ and $T$ are weakly
	compatible, $gz=gfz=gTw=Tgw=Tfz=Tz.$ Thus $z$ is a coincidence point of $g$
	and $T.$
	
	Now, $fx_{2n}\preceq x_{2n}$ and $x_{2n}\rightarrow z$ as $n\rightarrow
	\infty ,$ imply that $z\preceq fx_{2n}.$ Hence from (\ref{t.OMSeq1.1}), we get $\psi
	(p(fx_{2n},gz))\leq \psi \left( M_{p}(x_{2n},z)\right) -\varphi
	(M_{p}(x_{2n},z)),$ where
	\begin{eqnarray*}
		M_{p}(x_{2n},z) &=&\max \{p(Sx_{2n},Tz),p(fx_{2n},Sx_{2n}),p(gz,Tz),\frac{
			p(Sx_{2n},gz)+p(fx_{2n},Tz)}{2}\} \\
		&=&\max \{p(Sx_{2n},gz),p(fx_{2n},Sx_{2n}),p(gz,gz),\frac{
			p(Sx_{2n},gz)+p(fx_{2n},gz)}{2}\} \\
		&=&p(z,gz)\text{ as }n\rightarrow \infty .
	\end{eqnarray*}
	On taking upper limit as $n\rightarrow \infty ,$ we have $\psi (p(z,gz))\leq
	\psi \left( p(z,gz)\right) -\varphi (p(z,gz)),$ and $z=gz.$ Therefore $
	fz=gz=Sz=Tz=z.$ The proof is similar when $f$ is continuous.
	
	 Similarly, the result follows when \ref{condiiNaSeTheorem2.1.} holds.
	
	 Now, suppose that the set of common fixed points of $f$, $g$, $S$ and $T$ is
	well ordered. We are to show that the common fixed point of $f$, $g$, $S$
	and $T$ is unique. Suppose that $u$ and $v$ be two fixed points of $f,$ $g,$
	$S$\ and $T$ i.e., $fu=gu=Su=Tu=u$ and $fv=gv=Sv=Tv=v$ with $u\neq v.$ Then
	from (\ref{t.OMSeq1.1}), we have
	\begin{equation*}
		\psi (p(u,v))=\psi (p(fu,gv))\leq \psi (M_{p}(u,v))-\varphi (M_{p}(u,v)),
	\end{equation*}
	where
	\begin{eqnarray*}
		\text{ }M_{p}(u,v) &=&\max \{p(Su,Tv),p(fu,Su),p(gv,Tv),\dfrac{
			p(Su,gv)+p(fu,Tv)}{2}\} \\
		&=&\max \{p(u,v),p(u,u),p(v,v),\dfrac{p(u,v)+p(u,v)}{2}\} \\
		&=&p(u,v).
	\end{eqnarray*}
	Thus $\psi (p(u,v))\leq \psi (p(u,v))-\varphi (p(u,v)),$ a contradiction.
	Hence $u=v$. Conversely, if $f,$ $g,$\ $S$\ and $T$ have only one common
	fixed point then the set of common fixed point of $f,$ $g,$ $S\ $and $T$ is
	well ordered being singleton.
\qed \end{proof}

\begin{example}
	\label{NaSeExample2.2} Let $X=[0,k]$ for a real
	number $k\geq 9/10$ endowed with usual order $\leq .$ Let $p:X\times
	X\rightarrow \mathbb{R}_{\geq 0}$ be defined by $p(x,y)=\left \vert x-y\right
	\vert $ if $x,y\in \lbrack 0,1),\,$and $p(x,y)=\max \{x,y\}$ otherwise. It
	is easily seen that $(X,p)$ is a complete partial metric space \cite{AT}.
	Consider $\psi (t)=\left \{
	\begin{array}{l}
		3t,\text{ if }0\leq t\leq \frac{1}{3} \\
		1,\text{ \ if }x\in (\frac{1}{3},1] \\
		t,\text{ \ otherwise}%
	\end{array}%
	\right. $\ and $\varphi (t)=\left \{
	\begin{array}{l}
		0,\text{ \ if }t=0 \\
		\frac{t}{3},\text{ \ if }0<t\leq \frac{1}{3} \\
		\frac{1}{9},\text{ \ otherwise}%
	\end{array}%
	\right. .$ Define the self mappings $f$, $g$, $S$ and $T$ on $X$ by
	\begin{equation*}
		\begin{array}{ll}
			f(x)=\left \{
			\begin{array}{l}
				\frac{1}{6}x,\text{ if }x\leq \frac{1}{3} \\
				\frac{1}{18},\text{ if }x\in (\frac{1}{3},k]%
			\end{array}%
			\right. , & \text{ }gx=\left \{
			\begin{array}{l}
				0,\text{ if }x\leq \frac{1}{3} \\
				\frac{1}{3},\text{ if }x\in (\frac{1}{3},k]%
			\end{array}%
			\right. , \\
			T(x)=\left \{
			\begin{array}{l}
				0,\text{ \ if }x=0 \\
				x,\text{ \ if }x\in (0,\frac{1}{3}] \\
				k,\text{ \ if }x\in (\frac{1}{3},k]%
			\end{array}%
			\right. , & \text{ }Sx=\left \{
			\begin{array}{l}
				0,\text{ if }x=0 \\
				\frac{1}{3},\text{ if }x\in (0,\frac{1}{3}] \\
				k,\text{ if }x\in (\frac{1}{3},k]%
			\end{array}%
			\right. .%
		\end{array}%
	\end{equation*}
	Then $f(X)\subseteq T(X)$\ and $g(X)\subseteq S(X)$ with $f$ and $g$ are
	dominated and $S$ and $T$ are dominating mappings as%
	\begin{equation*}
		\begin{tabular}{|l|l|l|l|l|}
			\hline
			$\text{for each }x\text{ in }X$ & $fx\leq x$ & $gx\leq x$ & $x\leq Sx$ & $
			x\leq Tx$ \\ \hline
			$x=0$ & $f\left( 0\right) =0$ & $g\left( 0\right) =0$ & $0=S(0)$ & $0=T(0)$
			\\ \hline
			$x\in (0,\frac{1}{3}]$ & $fx=\frac{1}{6}x\leq x$ & $gx=0<x$ & $x\leq \frac{1%
			}{3}=S(x)$ & $x=T(x)$ \\ \hline
			$x\in (\frac{1}{3},k]$ & $fx=\frac{1}{18}<x$ & $gx=\frac{1}{3}<x$ & $x\leq
			k=S(x)$ & $x\leq k=T(x)$ \\ \hline
		\end{tabular}%
	\end{equation*}
	Also note that $\{f,S\}$\ are compatible, $\{g,T\}$\ are weakly compatible
	with $f$ is a continuous map.
	
 To show that $f,$ $g,$ $S$ and $T$ satisfy (\ref{t.OMSeq1.1}) for all $x,y\in X,$ we
	consider the following cases:	
	\begin{enumerate} [label=\textup{(\roman*)}, ref=(\roman*)]
		\item If $x=0\ $and $y\in \lbrack 0,\frac{1}{3}],$ then $p(fx,gy)=0$
		and (\ref{t.OMSeq1.1}) is satisfied.
		
		\item For $x=0\ $and $y\in (\frac{1}{3},k],$ we have
		\begin{eqnarray*}
			\psi (p(fx,gy)) &=&\psi (p(0,\frac{1}{3}))=\psi (\frac{1}{3}) \\
			&=&1<k-\frac{1}{9}=\psi (k)-\varphi (k) \\
			&=&\psi (p(0,k))-\varphi (p(0,k)) \\
			&=&\psi (p(Sx,Ty))-\varphi (p(Sx,Ty)) \\
			&=&\psi (M_{p}(x,y))-\varphi (M_{p}(x,y)).
		\end{eqnarray*}
		
		\item When $x=(0,\frac{1}{3}]\ $and $y\in \lbrack 0,\frac{1}{3}],$
		then
		\begin{eqnarray*}
			\psi (p(fx,gy)) &=&\psi (p(\frac{1}{6}x,0))=\psi (\frac{1}{6}x)=\frac{1}{2}x
			\\
			&\leq &3\max \{(\frac{1}{3}-\frac{1}{6}x),y\}-\frac{1}{3}\max \{(\frac{1}{3}-%
			\frac{1}{6}x),y\} \\
			&=&\psi (\max \{(\frac{1}{3}-\frac{1}{6}x),y\})-\varphi (\max \{(\frac{1}{3}-%
			\frac{1}{6}x),y\}) \\
			&=&\psi (\max \{p(fx,Sx),p(gy,Ty)\})-\varphi (\max \{p(fx,Sx),p(gy,Ty)\}) \\
			&=&\psi (M_{p}(x,y))-\varphi (M_{p}(x,y)).
		\end{eqnarray*}
		
		\item If $x=(0,\frac{1}{3}]\ $and $y\in (\frac{1}{3},k],$ then%
		\begin{eqnarray*}
			\psi (p(fx,gy)) &=&\psi (p(\frac{1}{6}x,\frac{1}{3}))=\psi (\frac{1}{3}(1-%
			\frac{x}{2}))=1-\frac{1}{2}x \\
			&<&k-\frac{1}{9}=\psi (\max \{ \frac{1}{3},k\})-\varphi (\max \{ \frac{1}{3}%
			,k\}) \\
			&=&\psi (p(gy,Ty))-\varphi (p(gy,Ty)) \\
			&=&\psi (M_{p}(x,y))-\varphi (M_{p}(x,y)).
		\end{eqnarray*}
		
		\item For $x\in (\frac{1}{3},k]\ $and $y\in \lbrack 0,\frac{1}{3}],$ we
		obtain%
		\begin{eqnarray*}
			\psi (p(fx,gy)) &=&\psi (p(\frac{1}{18},0))=\psi (\frac{1}{18})=\frac{1}{6}%
			<k-\frac{1}{9} \\
			&=&\psi (\max \{ \frac{1}{18},k\})-\varphi (\max \{ \frac{1}{18},k\}) \\
			&=&\psi (p(fx,Sx))-\varphi (p(fx,Sx)) \\
			&=&\psi (M_{p}(x,y))-\varphi (M_{p}(x,y)).
		\end{eqnarray*}
		
		\item Finally, when $x,y\in (\frac{1}{3},k],$ then we have
		\begin{eqnarray*}
			\psi (p(fx,gy)) &=&\psi (p(\frac{1}{18},\frac{1}{3}))=\psi (\frac{5}{18})=%
			\frac{5}{6}<k-\frac{1}{9} \\
			&=&\psi (\max \{ \frac{1}{3},k\})-\varphi (\max \{ \frac{1}{3},k\}) \\
			&=&\psi (p(gy,Ty))-\varphi (p(gy,Ty)) \\
			&=&\psi (M_{p}(x,y))-\varphi (M_{p}(x,y)).
		\end{eqnarray*}
	\end{enumerate}
	
	 The mappings $f,g,S$\ and $T$\ satisfy (\ref{t.OMSeq1.1}). Thus all the conditions given
	in Theorem \ref{NaSeTheorem2.1.} are satisfied. Moreover, \ $0$\ is the unique common fixed
	point of\ $f,$\ $g,$\ $S$\ and $T$. \end{example}

\begin{corollary}
	\label{NaSeCorollary2.3} Let $(X,$$\preceq ,p)$\textit{\ be an ordered complete partial metric space. Let
	}$f,g,S$\textit{\ and }$T$\textit{\ be self maps on }$X$\textit{, }$(f,g)$
	\textit{\ be the pair of dominated and }$(S,T)$\textit{\ be the pair of
		dominating maps with }$f\left( X\right) \subseteq T\left( X\right) $ and $
	g\left( X\right) \subseteq S\left( X\right) $\textit{. Suppose that, there
		exists control functions }$\psi $\textit{\ and }$\varphi $\textit{\ such
		that for every two comparable elements }$x,y\in X,$
	\begin{equation*}
		\psi (p(fx,gy))\leq \psi (p(Sx,Ty))-\varphi (p(Sx,Ty)) 
	\end{equation*}
	is satisfied.

If for any nonincreasing sequence $\{x_{n}\}$ in $(X,\preceq )$ with $x_{n}\preceq y_{n}$ for all $n$
	and $\lim \limits_{n\rightarrow \infty }p^{S}(x_{n},u)=0$  it holds that $u\preceq y_{n}$ for all $n\in\mathbb{Z}_{\geq 0}$, and either of the following conditions hold:
\begin{enumerate}[label=\textup{\arabic*)}, ref=\arabic*)]
		\item \label{condiNaSeCorollary2.3}
 $\{f,S\}$ are compatible, $f$ or $S$ is continuous on $(X,p^{S})$ and $\{g,T\}$ are weakly
compatible,
\item \label{condiiNaSeCorollary2.3} $\{g,T\}$ are compatible, $g$ or $T$ is continuous on $(X,p^{S})$ and $\{f,S\}$ are weakly
		compatible,
\end{enumerate}
then $f,g,S$ and $T$ have a common fixed point. Moreover, the set of common fixed points of $f$,
	$g$, $S$ and $T$ is well ordered if and only if $f$, $g$, $S$ and $T$
		have one and only one common fixed point.
\end{corollary}

Consistent with the terminology in \cite{Radenovic}, we denote $\Upsilon $
the set of all functions $\phi :
\mathbb{R}
_{\geq 0}\rightarrow
\mathbb{R}_{\geq 0},$ where $\phi $ is a Lebesgue integrable mapping with finite integral
on each compact subset of $
\mathbb{R}_{\geq 0},$ nonnegative, and for each $\varepsilon >0,$ $\int_{0}^{\varepsilon
}\phi (t)dt>0$ (see also, \cite{Branciari})$.$ As a consequence of Theorem
\ref{NaSeTheorem2.1.}, we obtain following fixed point result for a mapping satisfying contractive conditions of integral type in a complete partial metric space $X.$

\begin{corollary}
	\label{NaSeCorollary2.4.} \  \  Let $(X,\preceq
	, p)$ be an ordered complete partial metric space. Let $f,g,S$ and $T$ be self maps on $X$, $(f,g)$
	be the pair of dominated and $(S,T)$ be the pair of
		dominating maps with $f\left(X\right) \subseteq T\left( X\right)$ and
$g\left( X\right) \subseteq S\left( X\right)$. Suppose that, there
		exists control functions $\psi $ and $\varphi $ such
		that for every two comparable elements $x,y\in X,$
	\begin{equation}
		\int_{0}^{\psi (p(fx,gy))}\phi (t)dt\leq \int\limits_{0}^{\psi (M_{p}(x,y))}\phi
		(t)dt-\int\limits_{0}^{\varphi (M_{p}(x,y))}\phi (t)dt,  \label{t.OMSeq1.14}
	\end{equation}
	is satisfied, where $\phi \in \Upsilon $ and
	\begin{equation*}
		M_{p}(x,y)=\max \{p(Sx,Ty),p(fx,Sx),p(gy,Ty),\dfrac{p(Sx,gy)+p(fx,Ty)}{2}\}.
	\end{equation*}
	
If for any nonincreasing sequence $\{x_{n}\}$ in $
	(X,\preceq )$ with $x_{n}\preceq y_{n}$ for all $n$
	and $\lim \limits_{n\rightarrow \infty }p^{S}(x_{n},u)=0$, it holds that $u\preceq y_{n}$\textit{\ for all }$n\in
	\mathbb{Z}_{\geq 0}$, and either of the following conditions hold:
\begin{enumerate}[label=\textup{\arabic*)}, ref=\arabic*)]
		\item \label{condiNaSeCorollary2.4.}
 $\{f,S\}$\textit{\ are compatible, }$f$\textit{\ or }$S$ is continuous on $(X,p^{S})$ and $\{g,T\}$ are weakly compatible
\item \label{condiiNaSeCorollary2.4.} $\{g,T\}$ are compatible, $g$ or $T$ is continuous on $(X,p^{S})$ and $\{f,S\}$ are weakly
		compatible,
\end{enumerate}
then $f,g,S$ and $T$ have a common fixed point. Moreover, the set of common fixed points of $f,$
$g$, $S$ and $T$ is well ordered if and only if $f$, $g$, $S$ and $T$ have one and only one common fixed point.
	
\end{corollary}
\begin{proof} \smartqed Define $\Psi :\mathbb{R}
	_{\geq 0}\rightarrow \mathbb{R}
	_{\geq 0}$ by $\Psi (x)=\int\limits_{0}^{x}\phi (t)dt,$ then from (\ref{t.OMSeq1.14}), we have
	\begin{equation*}
		\Psi \left( \psi (p(fx,gy))\right) \leq \Psi \left( \psi (M_{p}(x,y))\right)
		-\Psi \left( \varphi (M_{p}(x,y))\right) ,
	\end{equation*}
	which can be written as
	\begin{equation*}
		\psi _{1}(p(fx,gy))\leq \psi _{1}(M_{p}(x,y))-\varphi _{1}(M_{p}(x,y)),
	\end{equation*}
	where $\psi _{1}=\Psi \circ \psi $ and $\varphi _{1}=\Psi \circ \varphi $.
	Clearly, $\psi _{1},\varphi _{1}:\mathbb{R}_{\geq 0}\rightarrow \mathbb{R}_{\geq 0},$ $
	\psi _{1}$ is continuous and nondecreasing, $\varphi _{1}$ is a lower
	semicontinuous, and $\psi _{1}(t)=\varphi _{1}(t)=0$ if and only if $t=0.$
	Hence by Theorem \ref{NaSeTheorem2.1.}, $f,g,S$ and $T$ have a unique common fixed point.
\qed \end{proof}

\begin{remark} We have the following remarks.
	\begin{enumerate}[label=\textup{\arabic*)}, ref=\arabic*]
		\item If we take $f=g$ and $S=T=I$ (an identity map) in Corollary \ref{NaSeCorollary2.3},
		then it extends \cite[Theorem 2.1]{Dutta} to ordered partial metric spaces.
		\item We can not apply Corollary \ref{NaSeCorollary2.3} in the setup of ordered metric
		space to the mappings given in Example \ref{NaSeExample2.2}. Indeed, if we take $x,y\in (%
		\frac{1}{3},2]$ contractive condition in the Corollary \ref{NaSeCorollary2.3} in the setup of
		ordered metric space is not satisfied.
		\item Theorem \ref{NaSeTheorem2.1.} generalizes \cite[Theorem 2.1]{AT},
		\cite[Theorem 2.1]{ATR} and \cite[Theorem 2.1]{Doric} for four maps in the setup of
		ordered partial metric spaces.
	\end{enumerate}
\end{remark}

\section{Application for Solutions of Implicit Integral Equations}

 Let $\Omega =[0,1]$ be bounded open set in $\mathbb{R}$, and $
L^{2}(\Omega )$ be the set of comparable functions on $\Omega $ whose square
is integrable on $\Omega .$\ Consider an integral equation%
\begin{equation}
	F(t,x(t))=\int \limits_{\Omega }\kappa (t,s,x(s))ds  \label{t.OMSeq1.16}
\end{equation}%
where $F:\Omega \times \mathbb{R}_{\geq 0}\rightarrow \mathbb{R}_{\geq 0}$ and $\kappa
:\Omega \times \Omega \times \mathbb{R}_{\geq 0}\rightarrow \mathbb{R}_{\geq 0}$ be two mappings. Feckan \cite{Fec} obtained the nonnegative solutions of
implicit integral equation (\ref{t.OMSeq1.16}) as an application of fixed point theorem.
We shall study the sufficient condition for existence of solution of
integral equation in framework of ordered complete partial metric space.
Define $p:X\times X\rightarrow
\mathbb{R}_{\geq 0}$ by
\begin{equation*}
	p(x,y)=\max \left( \underset{t\in \Omega }{\sup }\ x(t),\underset{t\in \Omega }{\sup }\ y(t)\right) .
\end{equation*}
Then $\left( X,p\right) $ is a complete partial metric space. We assume the following that there exists a positive
number $h\in \lbrack 0,\frac{1}{4})$:
\begin{enumerate}[label=\textup{(\roman*)}, ref=(\roman*)]	
\item $F(s,u(t))\leq h u(t)$ for each $s,t\in \Omega .$
\item $\int \limits_{\Omega }\kappa (t,s,v(s))ds\leq 2h v(t)\ $for
	each $s,t\in \Omega $.
\item The control functions\ $\psi $\ and $\varphi $\ are connected
	with relation that
	\begin{equation*}
		\psi (a)+\phi (2a)\leq \psi (2a),
	\end{equation*}
	for every $a\in
	\mathbb{R}_{\geq 0}$.
\end{enumerate}	
 Then integral equation (\ref{t.OMSeq1.16}) has a solution in $L^{2}(\Omega )$.

\begin{proof} \smartqed Define $(fx)(t)=F(t,x(t))$ and $
	(gx)(t)=\int \limits_{\Omega }\kappa (t,s,x(s))ds.$
Now
	\begin{eqnarray*}
		\psi (p(fx,gy)) &=&\psi \left( \max \left( \underset{t\in \Omega }{\sup }\
		\left( fx\right) (t),\underset{t\in \Omega }{\sup }\ \left( fy\right)
		(t)\right) \right) \\
		&=&\psi \left( \max \left( \underset{t\in \Omega }{\sup }\ F(t,x(t)),\underset{
			t\in \Omega }{\sup }\ \int \limits_{\Omega }\kappa (t,s,y(t))dt\right) \right)
		\\
		&\leq &\psi \left( \max \left( \underset{t\in \Omega }{\sup }\ hx(t),\underset{
			t\in \Omega }{\sup }\ 2hy(t)\right) \right) \\
		&\leq &\psi \left( 2h\max \left( \underset{t\in \Omega }{\sup }\ x(t),\underset%
		{t\in \Omega }{\sup }\ y(t)\right) \right) \\
		&\leq &\psi \left( \frac{1}{2}\max \left( \underset{t\in \Omega }{\sup }\ x(t),%
		\underset{t\in \Omega }{\sup }\ y(t)\right) \right) \\
		&=&\psi \left( \max \left( \underset{t\in \Omega }{\sup }\ x(t),\underset{t\in
			\Omega }{\sup }\ y(t)\right) \right) -\phi \left( \max \left( \underset{t\in
			\Omega }{\sup }\ x(t),\underset{t\in \Omega }{\sup }\ y(t)\right) \right) \\
		&=&\psi \left( p(x,y)\right) -\phi \left( p(x,y)\right) \\
		&=&\psi (M_{p}(x,y))-\varphi (M_{p}(x,y)).
	\end{eqnarray*}
	Thus for every comparable elements $x,y\in X,$
	\begin{equation*}
		\psi (p(fx,gy))\leq \psi (M_{p}(x,y))-\varphi (M_{p}(x,y)),
	\end{equation*}
	is satisfies where
	\begin{equation*}
		M_{p}(x,y)=\max \{p(x,y),p(fx,x),p(gy,y),\dfrac{p(fx,y)+p(gy,x)}{2}\}.
	\end{equation*}
	Now we can apply Theorem \ref{NaSeTheorem2.1.} by taking $S$ and $T$\ as identity maps to
	obtain the solution of integral equation (\ref{t.OMSeq1.16}) in $L^{2}(\Omega ).$
\qed \end{proof}

\section{Fractals in Partial Metric Spaces}

Consistent with \cite{AY}, let $CB^{p}(X)$ be the family of all non-empty,
closed and bounded subsets of the partial metric space $(X,p)$, induced by
the partial metric $p$. Note that closedness is taken from $(X,\tau _{p})$ ($
\tau _{p}$ is the topology induced by $p$) and boundedness is given as
follows: $A$ is a bounded subset in $(X,p)$ if there exists an $x_{0}\in X$
and $M\geq 0$ such that for all $a\in A,$ we have $a\in B_{p}(x_{0},M)$,
that is, $p(x_{0},a)<p(a,a)+M$. For $A,B\in CB^{p}(X)$ and $x\in X$, define $
\delta _{p}:CB^{p}(X)\times CB^{p}(X)\rightarrow \lbrack 0,\infty )$ and%
\begin{eqnarray*}
	p(x,A) &=&\inf \{p(x,a):a\in A\}, \\
	\delta _{p}(A,B) &=&\sup \{p(a,B):a\in A\}, \\
	H_{p}(A,B) &=&\max \{ \delta _{p}(A,B),\delta _{p}(B,A)\}.
\end{eqnarray*}
It can be verified that $p(x,A)=0$ implies $p^{S}(x,A)=0,$ where
$$p^{S}(x,A)=\inf \{p^{S}(x,a):a\in A\}.$$

\begin{lemma}[\cite{AL1}]
	\label{NaSeLemma4.1.} Let $(X,p)$ be a partial metric
	space and $A$ be a non-empty subset of $X,$ then $a\in \overline{A}$ if and
	only if $p(a,A)=p(a,a).$
\end{lemma}

\begin{proposition}[\cite{AY}]
	\label{NaSeProposition4.2.}  Let $(X,p)$ be a
	partial metric space. For any $A,B,C\in CB^{p}(X)$,
	\begin{enumerate}[label=\textup{(\roman*)}, ref=(\roman*)]
		\item $\delta _{p}(A,A)=\sup \{p(a,a):a\in A\};$	
		\item $\delta _{p}(A,A)\leq \delta _{p}(A,B);$	
		\item $\delta _{p}(A,B)=0$ implies $A\subseteq B;$	
		\item $\delta _{p}(A,B)\leq \delta _{p}(A,C)+\delta _{p}(C,B)-\inf
		\limits_{c\in C}p(c,c).$
	\end{enumerate}	
\end{proposition}

\begin{proposition}[\cite{AY}]
	\label{NaSeProposition4.3.}
	 Let $(X,p)$ be a
	partial metric space. For any $A,B,C\in CB^{p}(X),$
	\begin{enumerate}[label=\textup{\arabic*)}, ref=\arabic*]
		\item $H_{p}(A,A)\leq H_{p}(A,B);$
		\item $H_{p}(A,B)=H_{p}(B,A);$
		\item $H_{p}(A,B)\leq H_{p}(A,C)+H_{p}(C,B)-\inf \limits_{c\in C}p(c,c);$
		\item $H_{p}(A,B)=0$ implies that $A=B.$\smallskip
	\end{enumerate}
\end{proposition}

 The mapping $H_{p}:CB^{p}(X)\times CB^{p}(X)\rightarrow \lbrack 0,\infty )$\
is called partial Hausdorff metric induced by partial metric $p.$ Every
Hausdorff metric is partial Hausdorff metric but converse is not true (see
\cite[Example 2.6]{AY}).

\begin{theorem}[\cite{AY}]
	\label{NaSeTheorem4.4}  Let $(X,p)$ be a partial
	metric space. If $T:X\rightarrow CB^{p}(X)$ be a multi-valued mapping such
	that for all $x,y\in X,$ we have $H_{p}(Tx,Ty)\leq kp(x,y),$ where $k\in
	(0,1).$ Then $T$ has a fixed point.
\end{theorem}

\begin{definition}
	\label{NaSeDefinition4.5.} Let $(X,p)$\ be a partial metric space
	and and $\mathcal{H}_{p}(X)$ denotes the set of all non-empty compact
	subsets of $X.$ Let $\{f_{n}:n=1,\dots,N\}$\ be a finite family of
	self-mappings on $X$ that satisfy
	\begin{equation*}
		\psi (p(f_{i}x,f_{i}y))\leq \psi (M_{p}(x,y))-\varphi (M_{p}(x,y)),
	\end{equation*}
	where
	\begin{eqnarray*}
		M_{p}(x,y) &=&\max \{p(x,y),p(f_{i}x,x),p(f_{i}y,y),p(f_{i}^{2}x,f_{i}x), \\
		&&p(f_{i}^{2}y,y),p(f_{i}^{2}y,f_{i}y),\dfrac{p(f_{i}x,y)+p(f_{i}y,x)}{2}\}
	\end{eqnarray*}
	for every $x,y\in X.$ We call these maps as a family of generalized $\left(
	\psi ,\phi \right) $-contraction mappings. Define $T:\mathcal{H}%
	_{p}(X)\rightarrow \mathcal{H}_{p}(X)$ by
	\begin{eqnarray*}
		T(A) &=&f_{1}(A)\bigcup f_{2}(A)\bigcup \cdot \cdot \cdot \bigcup f_{N}(A) \\
		&=&\bigcup _{n=1}^{N}f_{n}(A),\text{ for each }A\in \mathcal{H}_{p}(X).
	\end{eqnarray*}
	If $f_{n}:X\rightarrow X$, $n=1,\dots,N$ are generalized $\left( \psi ,\phi
	\right)$-contraction mappings, then $(X;f_{1},f_{2},\dots,f_{N})$ is called
	generalized $\left( \psi ,\phi \right)$-iterated function system ($\left(\psi ,\phi \right)$-IFS).
\end{definition}

\begin{definition}
	\label{NaSedefinition4.6.} A nonempty compact set $A\subseteq X$
	is said to be an attractor of the generalized
$\left( \psi ,\phi \right)$-IFS if
\begin{enumerate}[label=\textup{\arabic*)}, ref=\arabic*]
	\item \label{defatr-i1TAeqA}$T(A)=A$ and
	\item \label{defatr-i2basinatr} there is an open set $U\subseteq X$ such that $A\subseteq U$ and $
		\lim \limits_{k\rightarrow \infty }T^{k}(B)=A$ for any compact set $
		B\subseteq U$, where the limit is taken with respect to the partial
		Hausdorff metric.
	\end{enumerate}
	
	The largest open set $U$ satisfying \ref{defatr-i2basinatr} is called a basin of
	attraction.
\end{definition}

\begin{theorem}
	\label{NaSeTheorem4.7} Let $(X,p)$\ be a complete partial
	metric space and $(X;f_{n},n=1,\dots,k)$ a generalized $\left( \psi ,\phi
	\right)$-iterated function system. Let $T:\mathcal{H}_{p}(X)\rightarrow
	\mathcal{H}_{p}(X)$ be a mapping defined by
	\begin{equation*}
		T(A)=\bigcup_{n=1}^{k}f_{n}(A),\text{ for all } A\in \mathcal{H}_{p}(X).
	\end{equation*}
	Suppose that, there exists control functions $\psi $\ and $\varphi $\ such
	that for every $A$, $B\in \mathcal{H}_{p}\left( X\right) ,$
	\begin{equation}
		\psi (H_{p}(T\left( A\right) ,T\left( B\right) ))\leq \psi (M_{T}(A,B))-\phi
		(M_{T}(A,B))  \label{t.OMSeq1.17}
	\end{equation}
	is satisfied, where
	\begin{eqnarray*}
		M_{T}(A,B) &=&\max \{H_{p}(A,B),H_{p}(A,T\left( A\right) ),H_{p}(B,T\left(
		B\right) ),H_{p}(T^{2}\left( A\right) ,T\left( A\right) ), \\
		&&H_{p}(T^{2}\left( A\right) ,B),H_{p}(T^{2}\left( A\right) ,T\left(
		B\right) ),\dfrac{H_{p}(A,T\left( B\right) )+H_{p}(B,T\left( A\right) )}{2}%
		\}.
	\end{eqnarray*}
	Then $T$ has a unique fixed point $U\in \mathcal{H}_{p}\left( X\right) ,$
	that is%
	\begin{equation*}
		U=T\left( U\right) =\bigcup _{n=1}^{k}f_{n}(U).
	\end{equation*}
	Moreover,\ for any initial set $A_{0}\in \mathcal{H}_{p}\left( X\right) $,
	the sequence $\{A_{0},T\left( A_{0}\right) ,T^{2}\left(
	A_{0}\right) ,\dots\}$ of compact sets converges to a fixed point of $T$.
\end{theorem}

\begin{proof} \smartqed Let $A_{0}\ $be an arbitrary element in $\mathcal{H}_{p}\left( X\right) .$ If $A_{0}=T\left( A_{0}\right) ,$ then the
	proof is finished. So we assume that $A_{0}\neq T\left( A_{0}\right) .$
	Define, for $m\in \mathbb{Z}_{\geq 0}$,  	
\begin{equation*}
A_{1}=T(A_{0}),\text{ }A_{2}=T\left( A_{1}\right) ,\dots,A_{m+1}=T\left(A_{m}\right).
\end{equation*}
	
	We may assume that $A_{m}\neq A_{m+1}$ for all $m\in
	\mathbb{Z}_{\geq 0}.$ If not, then $A_{k}=A_{k+1}$ for some $k$ implies $A_{k}=T(A_{k})\ $and
	this completes the proof. Take $A_{m}\neq A_{m+1}$ for all $m\in
	\mathbb{Z}_{\geq 0}
	$. From (\ref{t.OMSeq1.17}), we have
	\begin{eqnarray*}
		\psi \left( H_{p}(A_{m+1},A_{m+2})\right) &=&\psi \left( H_{p}(T\left(
		A_{m}\right) ,T\left( A_{m+1}\right) )\right) \\
		&\leq &\psi \left( M_{T}\left( A_{m},A_{m+1}\right) \right) -\phi \left(
		M_{T}\left( A_{m},A_{m+1}\right) \right) ,
	\end{eqnarray*}
	where
	\begin{eqnarray*}
		& M_{T}\left( A_{m},A_{m+1}\right) =\max \{H_{p}(A_{m},A_{m+1}),H_{p}\left(
		A_{m},T\left( A_{m}\right) \right) ,H_{p}\left( A_{m+1},T\left(A_{m+1}\right) \right) , \\
		&H_{p}(T^{2}\left( A_{m}\right) ,T\left( A_{m}\right) ),H_{p}\left(
		T^{2}\left( A_{m}\right) ,A_{m+1}\right) ,H_{p}\left( T^{2}\left(
		A_{m}\right) ,T\left( A_{m+1}\right) \right) , \\
		&\displaystyle{\frac{H_{p}\left( A_{m},T\left( A_{m+1}\right) \right) +H_{p}\left(
			A_{m+1},T\left( A_{m}\right) \right) }{2}}\} \\
		&=\max \{H_{p}(A_{m},A_{m+1}),H_{p}\left( A_{m},A_{m+1}\right) ,H_{p}\left(
		A_{m+1},A_{m+2}\right) , \\
		&H_{p}(A_{m+2},A_{m+1}),H_{p}\left( A_{m+2},A_{m+1}\right) ,H_{p}\left(
		A_{m+2},A_{m+2}\right) , \\
		&\displaystyle{\frac{H_{p}\left( A_{m},A_{m+2}\right) +H_{p}\left( A_{m+1},A_{m+1}\right)
		}{2}}\} \\
		&\leq \max \{H_{p}(A_{m},A_{m+1}),H_{p}\left( A_{m+1},A_{m+2}\right),
\displaystyle{\frac{H_{p}\left( A_{m},A_{m+1}\right) +H_{p}\left( A_{m+1},A_{m+2}\right)}{2}}\}
		\\
		&=\max \{H_{p}\left( A_{m},A_{m+1}\right) ,H_{p}\left(
		A_{m+1},A_{m+2}\right) \}.
	\end{eqnarray*}
	As $\max \{H_{p}\left( A_{m},A_{m+1}\right) ,H_{p}\left(
	A_{m+1},A_{m+2}\right) \} \leq M_{T}\left( A_{m},A_{m+1}\right).$ Therefore,
	$$M_{T}\left( A_{m},A_{m+1}\right) =\max \{H_{p}\left( A_{m},A_{m+1}\right)
	,H_{p}\left( A_{m+1},A_{m+2}\right) \}.$$
	
	Now if $M_{T}\left( A_{m},A_{m+1}\right) =H_{p}\left( A_{m+1},A_{m+2}\right)
	,$ then (\ref{t.OMSeq1.17}) gives that%
	\begin{equation*}
		\psi \left( H_{p}(A_{m+1},A_{m+2})\right) \leq \psi (H_{p}\left(
		A_{m+1},A_{m+2}\right) )-\phi \left( H_{p}\left( A_{m+1},A_{m+2}\right)
		\right) ,
	\end{equation*}
	a contradiction. Hence $M_{T}\left( A_{m},A_{m+1}\right) =H_{p}\left(
	A_{m+1},A_{m+2}\right) $ and%
	\begin{eqnarray*}
		\psi \left( H_{p}(A_{m+1},A_{m+2})\right) &\leq &\psi (H_{p}\left(
		A_{m},A_{m+1}\right) )-\phi \left( H_{p}\left( A_{m},A_{m+1}\right) \right)
		\\
		&\leq &\psi (H_{p}\left( A_{m},A_{m+1}\right) ),
	\end{eqnarray*}
	that is, $H_{p}\left( A_{m+1},A_{m+2}\right) \leq H_{p}\left(
	A_{m},A_{m+1}\right) .$ Thus the sequence $\{H_{p}\left(
	A_{m},A_{m+1}\right) \}$is nonincreasing. Hence there exists $c\geq 0$ such
	that $\lim \limits_{n\rightarrow \infty }H_{p}(A_{n},A_{n+1})=c.$ Suppose
	that $c>0.$ Then, $\psi (H_{p}(A_{n+2},A_{n+1}))\leq \psi
	(H_{p}(A_{n+1},A_{n}))-\varphi (H_{p}(A_{n+1},A_{n})),$ and by lower
	semicontinuity of $\varphi ,$ we have
	\begin{equation*}
		\underset{n\rightarrow \infty }{\lim \sup }\ \psi (H_{p}(A_{n+2},A_{n+1}))\leq
		\underset{n\rightarrow \infty }{\lim \sup }\ \psi (H_{p}(A_{n+1},A_{n}))-%
		\underset{n\rightarrow \infty }{\lim \inf }\ \varphi (H_{p}(A_{n+1},A_{n})),
	\end{equation*}
	which implies that $\psi (c)\leq \psi (c)-\varphi (c),$ a contradiction.
	Therefore $c=0.$ So we conclude that%
	\begin{equation}
		\lim_{n\rightarrow \infty }H_{p}(A_{n+1},A_{n})=0.  \label{t.OMSeq1.18}
	\end{equation}%
	Now, we show that $\lim \limits_{n,m\rightarrow \infty }H_{p}(A_{n},A_{m})=0$
	. If not, there is $\varepsilon >0,$ and there exist even integers $n_{k}$
	and $m_{k}$ with $m_{k}>n_{k}>k$ such that%
	\begin{equation}
		H_{p}(A_{m_{k}},A_{n_{k}})\geq \varepsilon ,  \label{t.OMSeq1.19}
	\end{equation}%
	and $H_{p}(A_{m_{k}-2},A_{n_{k}})<\varepsilon .$ Since
	\begin{eqnarray*}
		\varepsilon &\leq &H_{p}(A_{m_{k}},A_{n_{k}}) \\
		&\leq
		&H_{p}(A_{n_{k}},A_{m_{k}-2})+H_{p}(A_{m_{k}-2},A_{m_{k}})-\inf_{a_{1}\in
			A_{m_{k}-2}}p(a_{1},a_{1}) \\
		&\leq
		&H_{p}(A_{n_{k}},A_{m_{k}-2})+H_{p}(A_{m_{k}-2},A_{m_{k}-1})+H_{p}(A_{m_{k}-1},A_{m_{k}})
		\\
		&&-\inf_{a_{2}\in A_{m_{k}-1}}p(a_{2},a_{2})-\inf_{a_{1}\in
			A_{m_{k}-2}}p(a_{1},a_{1}).
	\end{eqnarray*}
	From (\ref{t.OMSeq1.18}) and (\ref{t.OMSeq1.19}), we have
	\begin{equation}
		\lim_{k\rightarrow \infty }H_{p}(A_{m_{k}},A_{n_{k}})=\varepsilon .
		\label{t.OMSeq1.20}
	\end{equation}%
	Also (\ref{t.OMSeq1.19}) and inequality $$H_{p}(A_{m_{k}},A_{n_{k}})\leq
	H_{p}(A_{m_{k}},A_{m_{k}-1})+H_{p}(A_{m_{k}-1},A_{n_{k}})-\inf
	\limits_{a_{2}\in A_{m_{k}-1}}p(a_{2},a_{2})$$ give that $\varepsilon \leq
	\lim \limits_{k\rightarrow \infty }H_{p}(A_{m_{k}-1},A_{n_{k}}),$ while
	inequality $$H_{p}(A_{m_{k}-1},A_{n_{k}})\leq
	H_{p}(A_{m_{k}-1},A_{m_{k}})+H_{p}(A_{m_{k}},A_{n_{k}})-\inf
	\limits_{a_{3}\in A_{m_{k}}}p(a_{3},a_{3})$$ yields
$\lim\limits_{k\rightarrow \infty }H_{p}(A_{m_{k}-1},A_{n_{k}})\leq \varepsilon ,$
	and hence
	\begin{equation}
		\lim_{k\rightarrow \infty }H_{p}(A_{m_{k}-1},A_{n_{k}})=\varepsilon .
		\label{t.OMSeq1.21}
	\end{equation}%
	Now (\ref{t.OMSeq1.21}) and inequality $$H_{p}(A_{m_{k}-1},A_{n_{k}})\leq
	H_{p}(A_{m_{k}-1},A_{n_{k}+1})+H_{p}(A_{n_{k}+1},A_{n_{k}})-\inf_{a_{4}\in
		A_{n_{k}+1}}p(a_{4},a_{4})$$ give $$\varepsilon \leq \lim
	\limits_{k\rightarrow \infty }H_{p}(A_{m_{k}-1},A_{n_{k}+1}),$$ while
	inequality $$H_{p}(A_{m_{k}-1},A_{n_{k}+1})\leq
	H_{p}(A_{m_{k}-1},A_{n_{k}})+H_{p}(A_{n_{k}},A_{n_{k}+1})-\inf
	\limits_{a_{5}\in A_{n_{k}}}p(a_{5},a_{5})$$ yields
$\lim\limits_{k\rightarrow \infty }H_{p}(A_{m_{k}-1},A_{n_{k}+1})\leq \varepsilon ,$ and so
	\begin{equation}
		\lim_{k\rightarrow \infty }H_{p}(A_{m_{k}-1},A_{n_{k}+1})=\varepsilon .
		\label{t.OMSeq1.22}
	\end{equation}%
	As%
	\begin{eqnarray*}
		M_{T}(A_{n_{k}},A_{m_{k}-1}) &=&\max
		\{H_{p}(A_{n_{k}},A_{m_{k}-1}),H_{p}(A_{n_{k}},A_{n_{k}}), \\
		&&H_{p}(A_{m_{k}-1},A_{m_{k}-1}),H_{p}(A_{n_{k}+2},A_{n_{k}}),H_{p}(A_{n_{k}+2},A_{m_{k}-1}),
		\\
		&&H_{p}(A_{n_{k}+2},A_{m_{k}}),\frac{
			H_{p}(A_{n_{k}},A_{m_{k}-1})+H_{p}(A_{n_{k}},A_{m_{k}-1})}{2}\}.
	\end{eqnarray*}
	So, $\lim \limits_{k\rightarrow \infty }M_{T}(x_{n_{k}},x_{m_{k}-1})=\max \{
	\varepsilon ,0,0,0,\varepsilon ,\varepsilon ,\varepsilon \}=\varepsilon .$
	From (\ref{t.OMSeq1.22}), we obtain%
	\begin{eqnarray*}
		\psi (H_{p}(A_{n_{k}+1},A_{m_{k}})) &=&\psi (H_{p}(A_{n_{k}},A_{m_{k}-1})) \\
		&\leq &\psi (M_{T}(A_{n_{k}},A_{m_{k}-1}))-\varphi
		(M_{T}(A_{n_{k}},A_{m_{k}-1})).
	\end{eqnarray*}
	Taking upper limit as $k\rightarrow \infty $ implies that $\psi (\varepsilon
	)\leq \psi (\varepsilon )-\varphi (\varepsilon ),$ a contradiction as $
	\varepsilon >0$. Therefore $\{A_{n}\}$ is a Cauchy sequence in $X.$ Since $(%
	\mathcal{H}_{p}(X),p)$ is complete as $(X,p)$ is complete, so $\lim
	\limits_{n\rightarrow \infty }H_{p}(A_{n},U)=H_{p}\left( U,U\right) $\ for
	some $U\in \mathcal{H}_{p}(X),$ that is, we have $A_{n}\rightarrow U$ as $
	n\rightarrow \infty .$
	
	In order to show that $U$ is the fixed point of $T,$ we contrary assume that
	$H_{p}\left( U,T\left( U\right) \right) \neq 0$. Now
	\begin{equation}
\begin{array}{ll}
		\psi \left( H_{p}(A_{n+1},T\left( U\right) )\right) & =\psi (H_{p}(T\left(
		A_{n}\right) ,T\left( U\right) ))\\
& \leq \psi \left( M_{T}\left( A_{n},U\right)
		\right) -\phi \left( M_{T}\left( A_{n},U\right) \right) ,
\end{array}
\label{t.OMSeq1.23}
	\end{equation}
	where
	\begin{eqnarray*}
		M_{T}\left( A_{n},U\right) &=&\max \{H_{p}(A_{n},U),H_{p}(A_{n},T\left(
		A_{n}\right) ),H_{p}(U,T\left( U\right) ),H_{p}(T^{2}\left( A_{n}\right)
		,T\left( A_{n}\right) ), \\
		&&H_{p}(T^{2}\left( A_{n}\right) ,U),H_{p}(T^{2}\left( A_{n}\right) ,T\left(
		U\right) ),\frac{H_{p}(A_{n},T\left( U\right) )+H_{p}(U,T\left( A_{n}\right)
			)}{2}\} \\
		&=&\max \{H_{p}(A_{n},U),H_{p}(A_{n},A_{n+1}),H_{p}(U,T\left( U\right)
		),H_{p}(A_{n+2},A_{n+1}), \\
		&&H_{p}(A_{n+2},U),H_{p}(A_{n+2},T\left( U\right) ),\frac{
			H_{p}(A_{n},T\left( U\right) )+H_{p}(U,A_{n+1})}{2}\}.
	\end{eqnarray*}
	Now we consider the following cases:
	\begin{enumerate}[label=\textup{\arabic*)}, ref=\arabic*]
		\item If $M_{T}\left( A_{n},U\right) =H_{p}(A_{n},U,),$ then on
	taking upper limit as $n\rightarrow \infty $ in (\ref{t.OMSeq1.23}), we have
	\begin{equation*}
		\psi \left( H_{p}(T\left( U\right) ,U)\right) \leq \psi \left( H_{p}\left(
		U,U\right) \right) -\psi \left( H_{p}\left( U,U\right) \right) ,
	\end{equation*}
	a contradiction.	
\item When $M_{T}\left( A_{n},U\right) =H_{p}(A_{n},A_{n+1}),$ then
	on taking upper limit as $n\rightarrow \infty $ in (\ref{t.OMSeq1.23}), implies
	\begin{equation*}
		\psi \left( H_{p}(T\left( U\right) ,U)\right) \leq \psi \left( H_{p}\left(
		U,U\right) \right) -\phi \left( H_{p}\left( U,U\right) \right) ,
	\end{equation*}
	gives a contradiction.	
\item In case $M_{T}\left( A_{n},U\right) =H_{p}(U,T\left( U\right)
	),$ then on taking upper limit as $n\rightarrow \infty $ in (\ref{t.OMSeq1.23}), we get
	\begin{equation*}
		\psi \left( H_{p}(T\left( U\right) ,U)\right) \leq \psi \left( H_{p}\left(
		U,T\left( U\right) \right) \right) -\phi \left( H_{p}\left( U,T\left(
		U\right) \right) \right) ,
	\end{equation*}
	a contradiction.	
\item If $M_{T}\left( A_{n},U\right) =\dfrac{H_{p}(A_{n},T\left(
		U\right) )+H_{p}(U,A_{n+1})}{2},$ then on upper taking limit as $
	n\rightarrow \infty ,$ we have
	\begin{multline*}
		\psi \left( H_{p}(T\left( U\right) ,U)\right) \leq \psi (\frac{H_{p}\left(
			U,T\left( U\right) \right) +H_{p}\left( U,U\right) }{2})\\
            -\phi (\frac{H_{p}\left( U,T\left( U\right) \right) +H_{p}\left( U,U\right) }{2}) \\
		=\psi (\frac{H_{p}\left( U,T\left( U\right) \right) }{2})-\phi (\frac{
			H_{p}\left( U,T\left( U\right) \right) }{2}),
	\end{multline*}
	a contradiction.	
\item When $M_{T}\left( A_{n},U\right) =H_{p}(A_{n+2},A_{n+1}),$
	then on taking upper limit as $n\rightarrow \infty $ in (\ref{t.OMSeq1.23}), we get
	\begin{equation*}
		\psi \left( H_{p}(T\left( U\right) ,U)\right) \leq \psi \left( H_{p}\left(
		U,U\right) \right) -\phi \left( H_{p}\left( U,U\right) \right) ,
	\end{equation*}
	gives a contradiction.	
\item In case $M_{T}\left( A_{n},U\right) =H_{p}(A_{n+2},U),$ then
	on taking upper limit as $n\rightarrow \infty $ in (\ref{t.OMSeq1.23}), we get
	\begin{equation*}
		\psi \left( H_{p}(T\left( U\right) ,U)\right) \leq \psi \left( H_{p}\left(
		U,U\right) \right) -\phi \left( H_{p}\left( U,U\right) \right) ,
	\end{equation*}
	a contradiction.	
\item Finally if $M_{T}\left( A_{n},U\right) =H_{p}(A_{n+2},T\left(
	U\right) ),$ then on taking upper limit as $n\rightarrow \infty ,$ we have
	\begin{equation*}
		\psi \left( H_{p}(T\left( U\right) ,U)\right) \leq \psi (H_{p}(U,T\left(
		U\right) ))-\phi \left( H_{p}\left( U,U\right) \right) ,
	\end{equation*}
	a contradiction.
	\end{enumerate}
	 Thus, $U$ is the fixed point of $T$.
	
	 To show the uniqueness of fixed point of $T$, assume that $U$ and $V$ are two fixed points of $T$ with $H_{p}\left( U,V\right) $ is not zero.
	From (\ref{t.OMSeq1.17}), we obtain that
	\begin{eqnarray*}
		\psi (H_{p}(U,V)) &=&\psi (H_{p}(T\left( U\right) ,T\left( V\right) )) \\
		&\leq &\psi \left( M_{T}\left( U,V\right) \right) -\phi \left( M_{T}\left(
		U,V\right) \right),
	\end{eqnarray*}
	where
	\begin{align*}
	&	M_{T}\left( U,V\right) = \max \{H_{p}(U,V),H_{p}(U,T\left( U\right)
		),H_{p}(V,T\left( V\right) ),\\
&  \frac{H_{p}(U,T\left( V\right))+H_{p}(V,T\left( U\right) )}{2}, H_{p}(T^{2}\left( U\right) ,U),H_{p}(T^{2}\left( U\right)
		,V),H_{p}(T^{2}\left( U\right) ,T\left( V\right) )\} \\
		&=\max \{H_{p}\left( U,V\right) ,H_{p}(U,U),H_{p}(V,V),
        \\
        & \hspace{2cm} \frac{H_{p}(U,V)+H_{p}(V,U)}{2},H_{p}\left( U,U\right) ,H_{p}(U,V),H_{p}(U,V)\} \\
		&=H_{p}(U,V),
	\end{align*}
	that is,
	\begin{equation*}
		\psi (H_{p}(U,V))\leq \psi \left( H_{p}\left( U,V\right) \right) -\phi
		\left( H_{p}\left( U,V\right) \right) ,
	\end{equation*}
	a contradiction. Thus $T$ has a unique fixed point $U\in \mathcal{H}_{p}(X)$.
\qed \end{proof}

\begin{remark}
	\label{NaSeRemark4.8.} In Theorem \ref{NaSeTheorem4.7}, if we take $\mathcal{S}
	(X) $ the collection of all singleton subsets of $X,$ then clearly
$\mathcal{S}(X)\subseteq \mathcal{H}_{p}(X).$ Moreover, consider $f_{n}=f$ for each $n,$ where $f=f_{1}$ then the mapping $T$ becomes
	\begin{equation*}
		T(x)=f(x).
	\end{equation*}
\end{remark}

With this setting, we obtain the following fixed point result.

\begin{corollary}
	\label{NaSeCorollary4.9.}
	Let $(X,p)$ be a complete partial
	metric space and $\{X:f_{n},n=1,\dots, k\}$ a generalized iterated function system.
Let $f:X\rightarrow X$ be a mapping defined as in Remark \ref{NaSeRemark4.8.}. Suppose that, there
	exists control functions $\psi $ and $\varphi $ such that for any $x,y\in
	\mathcal{H}_{p}\left( X\right) $, the following holds:
	\begin{equation*}
		\psi \left( p\left( fx,fy\right) \right) \leq \psi (M_{p}(x,y))-\phi
		(M_{p}(x,y)),
	\end{equation*}
	where
	\begin{eqnarray*}
		M_{p}(x,y)=&&\max \{
		p(x,y),p(x,fx),p(y,fy),p(f^{2}x,y),\\
		&&p(f^{2}x,fx),p(f^{2}x,fy),\dfrac{p(x,fy)+p(y,fx)}{2} \} .
\end{eqnarray*}
	Then $f$ has a unique fixed point $x\in X,$ Moreover, for any initial set $
	x_{0}\in X$, the sequence of compact sets $\{x_{0},fx_{0},f^{2}x_{0},\dots\}$
	converges to a fixed point of $f$.
\end{corollary}

\begin{corollary}
	\label{NaSeCorollary4.10.}
	Let $(X,p)$ be a complete partial
	metric space and $(X;f_{n},n=1,\dots, k)$ be iterated function system where each $f_{i}$ for
	$i=1,\dots,k$ is a contraction self-mapping on $X$. Then $T:\mathcal{H}(X)\rightarrow \mathcal{H}(X)$ defined in Theorem \ref{NaSeTheorem4.7} has a unique fixed
	point in $\mathcal{H}\left( X\right) .$ Furthermore, for any set $A_{0}\in
	\mathcal{H}\left( X\right) $, the sequence of compact sets $\{A_{0},T\left(A_{0}\right) ,T^{2}\left( A_{0}\right) ,\dots\}$ converges to a fixed point of
	$T$.
\end{corollary}

\begin{proof} \smartqed
	It follows from Theorem \ref{NaSeTheorem1.9.} that if each $f_{i}$ for $i=1,\dots,k$ is a contraction mapping on $X,$ then the mapping
$T:\mathcal{H}(X)\rightarrow \mathcal{H}(X)$ defined by
	\begin{equation*}
		T(A)=\bigcup _{n=1}^{k}f_{n}(A),\text{ for all }A\in \mathcal{H}(X)
	\end{equation*}
	is contraction on $\mathcal{H}\left( X\right) $. Using Theorem \ref{NaSeTheorem2.1.}, the
	result follows.
\qed \end{proof}

\begin{corollary}
	\label{NaSeCorollary4.11.} Let $(X,p)$ be a complete partial metric space and $(X;f_{n},n=1,\dots,k)$ an iterated function system where each $f_{i}$ for
	$i=1,\dots,k$ is a mapping on $X$ satisfying
	\begin{equation*}
		d\left( f_{i}x,f_{i}y\right) e^{d\left( f_{i}x,f_{i}y\right) -d\left(
			x,y\right) }\leq e^{-\tau }d\left( x,y\right) ,\text{ for all }x,y\in X,
		\text{ }f_{i}x\neq f_{i}y,
	\end{equation*}
	where $\tau >0.$ Then the mapping $T:\mathcal{H}(X)\rightarrow \mathcal{H}
	(X) $ defined in Theorem \ref{NaSeTheorem4.7} has a unique fixed point in $\mathcal{H}\left(
	X\right) .\ $Furthermore, for any set $A_{0}\in \mathcal{H}\left( X\right) $
	, the sequence of compact sets $\{A_{0},T\left( A_{0}\right) ,T^{2}\left(
	A_{0}\right) ,\dots\}$ converges to a fixed point of $T$.
\end{corollary}

\begin{proof} \smartqed
	Take $F\left( \lambda \right) =\ln \left(
	\lambda \right) +\lambda ,$ $\lambda >0$ in Theorem \ref{NaSeTheorem1.11}, then each mapping $
	f_{i}$ for $i=1,\dots,k$ on $X$ satisfies%
	\begin{equation*}
		d\left( f_{i}x,f_{i}y\right) e^{d\left( f_{i}x,f_{i}y\right) -d\left(
			x,y\right) }\leq e^{-\tau }d\left( x,y\right) ,\text{ for all }x,y\in X,
		\text{ }f_{i}x\neq f_{i}y,
	\end{equation*}
	where $\tau >0.$ Again from Theorem \ref{NaSeTheorem1.11}, the mapping $T:\mathcal{H}
	(X)\rightarrow \mathcal{H}(X)$ defined by
	\begin{equation*}
		T(A)=\bigcup _{n=1}^{k}f_{n}(A),\text{ for all }A\in \mathcal{H}(X)
	\end{equation*}
	satisfies%
	\begin{equation*}
		H\left( T\left( A\right) ,T\left( B\right) \right) e^{H\left( T\left(
			A\right) ,T\left( B\right) \right) -H\left( A,B\right) }\leq e^{-\tau
		}H\left( A,B\right),
	\end{equation*}
	for all $A,B\in \mathcal{H}(X),$ $H\left( T\left( A\right) ,T\left( B\right)
	\right) \neq 0$. Using Theorem \ref{NaSeTheorem2.1.}, the result follows.
\qed \end{proof}

\begin{corollary}
	\label{NaSeCorollary4.12.}
	Let $(X,p)$ be a complete partial
	metric space and $(X;f_{n},n=1,\dots, k)$ be iterated function system such that each $f_{i}$
	for $i=1,\dots,k$ is a mapping on $X$ satisfying
	\begin{equation*}
		d\left( f_{i}x,f_{i}y\right) (d\left( f_{i}x,f_{i}y\right) +1)\leq e^{-\tau
		}d\left( x,y\right) (d\left( x,y\right) +1),\text{ for all }x,y\in X,\
		f_{i}x\neq f_{i}y,
	\end{equation*}
	where $\tau >0.$ Then the mapping $T:\mathcal{H}(X)\rightarrow \mathcal{H}%
	(X) $ defined in Theorem \ref{NaSeTheorem2.1.} has a unique fixed point in $\mathcal{H}\left(
	X\right).$ Furthermore, for any set $A_{0}\in \mathcal{H}\left( X\right) $,
	the sequence of compact sets $\{A_{0},T\left( A_{0}\right) ,T^{2}\left(
	A_{0}\right) ,\dots\}$ converges to a fixed point of $T$.
\end{corollary}
\begin{proof} \smartqed By taking $F\left( \lambda \right) =\ln
	\left( \lambda ^{2}+\lambda \right) +\lambda ,$ $\lambda >0$ in Theorem
	\ref{NaSeTheorem1.11}, we obtain that each mapping $f_{i}$ for $i=1,\dots,k$ on $X$ satisfies
	\begin{equation*}
		d\left( f_{i}x,f_{i}y\right) (d\left( f_{i}x,f_{i}y\right) +1)\leq e^{-\tau
		}d\left( x,y\right) (d\left( x,y\right) +1),\text{ for all }x,y\in X,\
		f_{i}x\neq f_{i}y,
	\end{equation*}
	where $\tau >0.$ Again it follows from Theorem \ref{NaSeTheorem1.11} that the mapping $T:%
	\mathcal{H}(X)\rightarrow \mathcal{H}(X)$ defined by
	\begin{equation*}
		T(A)=\bigcup_{n=1}^{k}f_{n}(A),\text{ for all }A\in \mathcal{H}(X)
	\end{equation*}
	satisfies
	\begin{equation*}
		H\left( T\left( A\right) ,T\left( B\right) \right) (H\left( T\left( A\right)
		,T\left( B\right) \right) +1)\leq e^{-\tau }H\left( A,B\right) (H\left(
		A,B\right) +1),
	\end{equation*}
	for all $A,B\in \mathcal{H}(X),$ $H\left( T\left( A\right) ,T\left( B\right)
	\right) \neq 0$. Using Theorem \ref{NaSeTheorem2.1.}, the result follows.
\qed \end{proof}

\begin{corollary}
	\label{NaSeCorollary4.13.} Let $(X,p)$ be a complete partial
	metric space and $(X;f_{n},n=1,\dots, k)$ be iterated function system such that each $f_{i}$
	for $i=1,\dots,k$ is a mapping on $X$ satisfying
	\begin{equation*}
		d\left( f_{i}x,f_{i}y\right) \leq \frac{1}{(1+\tau \sqrt{d\left( x,y\right)})}d\left( x,y\right) ,\text{ for all }x,y\in X,\text{ }f_{i}x\neq f_{i}y,
	\end{equation*}
	where $\tau >0.$ Then the mapping $T:\mathcal{H}(X)\rightarrow \mathcal{H}
	(X) $ defined in Theorem \ref{NaSeTheorem4.7} has a unique fixed point $\mathcal{H}\left(
	X\right).$ Furthermore, for any set $A_{0}\in \mathcal{H}\left( X\right) $,
	the sequence of compact sets $\{A_{0},T\left( A_{0}\right) ,T^{2}\left(
	A_{0}\right) ,\dots\}$ converges to a fixed point of $T$.
\end{corollary}

\begin{proof} \smartqed Take $F\left( \lambda \right) =-1/\sqrt{
		\lambda },$ $\lambda >0$ in Theorem \ref{NaSeTheorem1.11}, then each mapping $f_{i}$ for $
	i=1,\dots,k$ on $X$ satisfies
	\begin{equation*}
		d\left( f_{i}x,f_{i}y\right) \leq \frac{1}{(1+\tau \sqrt{d\left( x,y\right) }
			)^{2}}d\left( x,y\right) ,\text{ for all }x,y\in X,\ f_{i}x\neq f_{i}y,
	\end{equation*}
	where $\tau >0.$ Again it follows from Theorem \ref{NaSeTheorem1.11} that the mapping
$T: \mathcal{H}(X)\rightarrow \mathcal{H}(X)$ defined by
	\begin{equation*}
		T(A)=\bigcup\limits_{n=1}^{k}f_{n}(A),\text{ for all }A\in \mathcal{H}(X)
	\end{equation*}
	satisfies%
	\begin{equation*}
		H\left( T\left( A\right) ,T\left( B\right) \right) \leq \frac{1}{(1+\tau
			\sqrt{H\left( A,B\right) })^{2}}H\left( A,B\right),
	\end{equation*}
	for all $A,B\in \mathcal{H}(X),$ $H\left( T\left( A\right), T\left( B\right)
	\right) \neq 0$. Using Theorem \ref{NaSeTheorem2.1.}, the result follows.
\qed \end{proof}



\end{document}